\documentclass[sn-mathphys]{sn-jnl}

\jyear{2021}%

\theoremstyle{thmstyleone}%
\newtheorem{theorem}{Theorem}
\newtheorem{lemma}[theorem]{Lemma}%
\newtheorem{corollary}[theorem]{Corollary}%

\theoremstyle{thmstyletwo}%
\newtheorem{example}{Example}%

\theoremstyle{thmstylethree}%

\begin{document}

\title[QC-LDPC Codes from Difference Matrices and Difference Covering Arrays]{QC-LDPC Codes from Difference Matrices and Difference Covering Arrays}


\author[1]{\fnm{Diane} \sur{Donovan}}\email{dmd@maths.uq.edu.au}
\equalcont{These authors contributed equally to this work.}

\author[2]{\fnm{Asha} \sur{Rao}}\email{asha@rmit.edu.au}
\equalcont{These authors contributed equally to this work.}

\author[3]{\fnm{Elif} \sur{\"{U}sk\"{u}pl\"{u}}}\email{uskupluelif@gmail.com }
\equalcont{These authors contributed equally to this work.}

\author*[4]{\fnm{E. \c{S}ule} \sur{Yaz\i c\i}}\email{eyazici@ku.edu.tr }
\equalcont{These authors contributed equally to this work.}

\affil[1]{\orgdiv{Mathematics}, \orgname{The University of Queensland}, \orgaddress{\street{St Lucia}, \city{Brisbane}, \state{QLD}, \country{Australia}}}

\affil[2]{\orgdiv{Mathematical Sciences}, \orgname{RMIT Univerity}, \orgaddress{\city{Melbourne}, \postcode{3000}, \country{Australia}}}

\affil[3]{\orgdiv{Mathematics}, \orgname{Ko\c{c} University}, \orgaddress{\street{Sar\i yer}, \city{Istanbul}, \postcode{34450}, \country{Turkey}}}

\affil*[4]{\orgdiv{Mathematics}, \orgname{Ko\c{c} University}, \orgaddress{\street{Sar\i yer}, \city{Istanbul}, \postcode{34450}, \country{Turkey}}}

\abstract{We give a framework for generalizing LDPC code constructions that use Transversal Designs or related structures such as mutually orthogonal Latin squares. Our construction offers a broader range of code lengths and codes rates. Similar earlier constructions rely on the existence of finite fields of order a power of a prime. In contrast the LDPC codes  constructed here are based on difference matrices and difference covering arrays, structures available for any order $a$. They satisfy the RC constraint and have, for $a$ odd, length $a^2$ and rate $1-\frac{4a-3}{a^2}$,  and  for $a$ even, length $a^2-a$ and rate at least $1-\frac{4a-6}{a^2-a}$. When $3$ does not divide $a$, these LDPC codes have stopping distance at least $8$. When $a$ is odd and both $3$ and $5$ do not divide $a$, our construction delivers an infinite family of QC-LDPC codes with minimum distance at least $10$.
The simplicity of the construction allows us to theoretically verify these properties and analytically determine lower bounds for the minimum distance and stopping distance of the code.
The BER and FER performance of our codes over AWGN (via simulation) is at the least equivalent to codes constructed previously, while in some cases significantly outperforming them.}

\keywords{LDPC codes, QC-LDPC codes, combinatorial construction, difference matrices, difference covering arrays}



\maketitle

\section{Introduction}\label{sec1}

The roll-out of smart devices for IoT and 5G networks necessitate the development of efficient techniques  maximizing the integrity of data  sent or received through open channels, where the data may be subject to distortion, attenuation and Gaussian noise. Error correction codes are being developed to meet these needs, where these codes are designed to significantly enhance the reliability and integrity of transmitted data. While turbo codes have been implemented in smart 3G and 4G devices, the current demand for massive machine type communication, with ultra-reliability and low latency, is much higher, with 5G new radio (NR) requirements reaching through-puts of 5Gb/s. To meet this challenge researchers are investigating the use of LDPC (Low Density Parity Check) and polar codes, see   \cite{Bae2019, RichardsonKudekar2018, StarkBauchetal2020}.
In a 5G network, functionality requirements for control of both information and user data  indicate the need for codes that support variable code rates and lengths \cite{Bae2019}. In addition, storage and computational power can be restricted in modern smart devices, necessitating the development of codes based on low density or sparse cyclically generated parity-check matrices that can deliver low decoding complexity and enable parallelism in encoding and decoding.  LDPC codes have been shown to meet these requirements by delivering effective tools  compatible with 5G encoding and decoding, incorporating variable code lengths and code rates to meet the demands of 5G user data, \cite{Bae2019}.

   Randomly constructed LDPC codes were first introduced by Gallager in 1962 \cite{Gallager62} with MacKay and Neal later showing  that these LDPC codes are able to achieve rates close to channel capacity \cite{MacKayNeal97}. However randomly generated LDPC codes can lead to high storage overheads with complex implementation routines. Thus there is a need for LDPC codes having a compact representation with low storage requirements, that also support efficient encoding and decoding algorithms  \cite{Lally2007}. To address this need, a number of authors \cite{ParkHongNoShin2013,LiLiLi2017,VasMil04,kamiya2007high,KouLinFoss01,Zhang10},
 have proposed constructing  quasi-cyclic parity-check matrices for LDPC codes from combinatorial
 structures such as perfect cyclic difference sets, transversal designs,
  block designs and finite fields. However, the existence of these underlying
   algebraic and combinatorial structures is
   generally restricted to orders a power of a prime, making it difficult to achieve the highly desirable property of flexibility in code lengths and rates.

 In the current paper, gains are made by developing a construction based on cyclically generated orthogonal Latin squares
 that works over the cyclic group of order $a$, where the operation is addition modulo $a$,
      exploiting the fact that cyclic groups  exist for all orders  $a$. The cyclic nature of the  proposed construction provides for reduced storage and enables parallelism in encoding and decoding with increased options for code lengths and rates together with control over other code parameters such as girth and minimum distance.

      Further flexibility is obtained by utilising the combinatorial properties of ubiquitous
difference (covering) arrays, as opposed to, for example, less prevalent perfect cyclic difference sets \cite{LiLiLi2017} or transversal designs. If even greater flexibility is sought, a difference (covering) arrays may be   defined over any  abelian group. In addition, we show through simulations that this  greater range of code length and rates is not at the expense of performance, with the constructed codes performing equal to or better than other codes constructed using similar constructions.

We begin with the requisite coding  theory definitions and background in the next Section,  
going on to define difference matrices and difference covering arrays and the proposed constructions in Section \ref{DCA}. Determination of rates and other properties for the LDPC and QC-LDPC (quasi-cyclic LDPC) codes constructed here are given in Section \ref{properties}, with a  performance analysis given in Section \ref{PA} and concluding remarks in Conclusion Section.

\section{Background}\label{background}
We start with the preliminary definitions.

A $(m,w_c, w_r)$-{\em regular binary LDPC code} ${\cal C}$  of block length $m$ is given by the null
space of an $x\times m$  sparse $(0,1)$ parity-check matrix $H=[H(i,j)]$,  where both the row weight $w_r$ and column weight $w_c$  are constant, see \cite{Gallager62}. Given a parity-check matrix $H=[H(i,j)]$, an $m$-tuple ${\bf v}=(v_0,v_1,v_2,\dots, v_{m-1})$ is  a code word if and only if the  syndrome ${\bf S}$, shown in Equation \eqref{syndrome}, is the zero vector.
\begin{eqnarray}
{\bf S}=\left[\sum_{j =0}^{m-1} H(i,j) v_j\mbox{ mod }2\right], \mbox{ where }0\leq i\leq  x-1.\label{syndrome}
\end{eqnarray}
Note that since the parity-check matrix is binary, we work over ${\Bbb Z}_2$. Also in this paper the rows and columns of all $x\times m$ matrices will be indexed by the set $\{0,1,\dots, x-1\}$ and $\{0,1,\dots, m-1\}$ respectively.

If, after row reduction, the parity-check matrix can be written in the form
\begin{eqnarray*}
H=\left[ P^T\mid I_{m-\kappa}\right],
\end{eqnarray*}
then low density manifests as $w_r\ll m$ and $w_c\ll m-\kappa$. The {\em rate of the code} is defined to be  $ \kappa/m$.
A parity-check matrix $H$ is said to satisfy the {\em RC-constraint} if the inner product of any two rows and any two columns is at most one.
 The {\em distance} of the code is taken to be the minimum Hamming distance between any two distinct code words. Since the code is linear and the zero vector is a code word the distance of the code  is equal to the minimum weight over all the non-zero code words.

In this paper, we seek to construct parity-check matrices that provide
good variability in the  code length and rate while maintaining  the minimum distance of the code to be at least $8$ and at least $10$ for certain cases.
For even $m$, this is achieved  by relaxing the regularity condition.

  We define a  $(m,w_c, \{w_r-1,w_r\})$-{\em near regular binary LDPC code} ${\cal C}$  of block length $m$ as the null
space of a  sparse $(0,1)$ parity-check matrix $H=[H(i,j)]$, with column weight $w_c$ and varying
row weights  $w_r-1$ or $w_r$.
\begin{example}\label{ex-PE}
 The following matrix is an example of a parity-check matrix  for  a $(12,4,\{4,3\})$-near regular binary LDPC code. The length of the code is $12$, the column weight is $w_c=4$ and the row weight is $w_r=3$ or $4$. Minimum distance is $8$. The rank of ${\overline H}_{12}$ is $10$ so the rate of the code is $2/12=0.17$.  This parity-check matrix satisfies the RC-constraint.
\begin{eqnarray*}
{\overline H}_{12}&=&\left[
\begin{array}{cccc|cccc|cccc}
1 & 1 & 1 & 1 & 0 & 0 & 0 & 0 & 0 & 0 & 0 & 0 \\
0 & 0 & 0 & 0 & 1 & 1 & 1 & 1 & 0 & 0 & 0 & 0 \\
0 & 0 & 0 & 0 & 0 & 0 & 0 & 0 & 1 & 1 & 1 & 1 \\
\hline
1 & 0 & 0 & 0 & 1 & 0 & 0 & 0 & 1 & 0 & 0 & 0 \\
0 & 1 & 0 & 0 & 0 & 1 & 0 & 0 & 0 & 1 & 0 & 0 \\
0 & 0 & 1 & 0 & 0 & 0 & 1 & 0 & 0 & 0 & 1 & 0 \\
0 & 0 & 0 & 1 & 0 & 0 & 0 & 1 & 0 & 0 & 0 & 1 \\
\hline
1 & 0 & 0 & 0 & 0 & 0 & 0 & 1 & 0 & 1 & 0 & 0 \\
0 & 1 & 0 & 0 & 1 & 0 & 0 & 0 & 0 & 0 & 1 & 0 \\
0 & 0 & 1 & 0 & 0 & 1 & 0 & 0 & 0 & 0 & 0 & 1 \\
0 & 0 & 0 & 1 & 0 & 0 & 1 & 0 & 1 & 0 & 0 & 0 \\
\hline
0 & 0 & 0 & 1 & 0 & 1 & 0 & 0 & 0 & 0 & 1 & 0 \\
1 & 0 & 0 & 0 & 0 & 0 & 1 & 0 & 0 & 0 & 0 & 1 \\
0 & 1 & 0 & 0 & 0 & 0 & 0 & 1 & 1 & 0 & 0 & 0 \\
0 & 0 & 1 & 0 & 1 & 0 & 0 & 0 & 0 & 1 & 0 & 0 \\
\end{array}
\right]
\end{eqnarray*}
\end{example}
In the above example the code length is relatively small, but the minimum distance is relatively high with respect to the code length $m$.

The construction proposed here follows the principles as set out in  Gallager's 1962 paper, \cite{Gallager62}, with the parity-check matrices for LDPC codes  constructed by
combining   submatrices, with  each column of each submatrix a cyclic shift of the previous. Gallager  defined the first submatrix   and then applied random permutations to the columns of this submatrix to obtain the remaining submatrices.
However, the randomization of the submatrices increases the storage costs resulting in less memory-efficient codes, see \cite{Aror2019}.  To avoid these storage issues  quasi-cyclic
 LDPC codes, or QC-LDPC codes, have been proposed.

 The parity-check
matrix $H$ for a QC-LDPC code can be written as a $K \times L$ array of $z \times z$ circulant matrices $H_{(i,j)}$,
where each circulant $H_{(i,j)}$ for $1\leq i\leq K$, $1\leq j\leq L$, is a square matrix with each row a cyclic shift of the previous. Hence $H_{(i,j)}$ is the zero matrix, a circulant permutation matrix, or the sum of $1\leq \lambda\leq z$ \textit{disjoint} circulant permutation matrices. Adhering to the general framework as set out by Gallager in 1962, we will specify the parity-check matrix in terms of related submatrices.

  Thus the general structure of the parity-check matrices is
\begin{eqnarray}\label{eq:QC}
H &=& \left[
\begin{array}{llll}
H_{(1,1)}& H_{(1,2)}&...&H_{(1,L)}\\
H_{(2,1)}& H_{(2,2)}&...&H_{(2,L)}\\
...\\
H_{(K,1)} &H_{(K,2)}&...&H_{(K,L)}\\
\end{array}
\right].
\end{eqnarray}
Historically, algebraic or combinatorial techniques have been used to specify the submatrices $H_{(i,j)}$ with this compact mathematical representation enhancing the encoding algorithms and minimizing the storage requirements, while maintaining  low computational complexity when implemented \cite{LiLiLi2017}.

In this paper, we will first define our codes to be quasi-cyclic ``like'' in that, cyclic shifts of any code word within each subblock will also be a code word. Then we will show that some infinite subclasses of these parity-check matrices can  be rearranged using row and column permutations to obtain the quasi-cyclic form. The constructed codes with quasi-cyclic structure will be example of codes with Tanner graphs that are cyclic liftings of fully connected base graphs of size $4\times a$ with a lifting factor of $a$. Refer to \cite{Tasdighi} and the references therein  for definitions and related results.

To this end, let $a>3$ be a positive integer. Define  $H$ to be a $(4a)\times (a^2)$ matrix
 \begin{eqnarray}
 H=\left[
 \begin{array}{cccc}
 R_0& R_1& \dots& R_{a-1}\\
 \mathfrak{P}_{0,0} & \mathfrak{P}_{0,1} &\dots & \mathfrak{P}_{0,a-1}\\
 \mathfrak{P}_{1,0}&\mathfrak{P}_{1,1}&\dots &\mathfrak{P}_{1,a-1}\\
 \mathfrak{P}_{2,0}&\mathfrak{P}_{2,1}&\dots &\mathfrak{P}_{2,a-1}\\
 \end{array}\right]\label{eq:H0}
 \end{eqnarray}
 where
 \begin{itemize}
\item[-]  $R_v=[R_v(i,j)]$ is taken to be an $a\times a$ square matrix with row $v$  the vector of all one's and every other row the vector of all zeros.
 \item[-]  for $0\leq u\leq 2$ and $0\leq v\leq a-1$, $\mathfrak{P}_{u,v}$  is taken to be a permutation of the $a\times a$ identity matrix, denoted $I$.
\end{itemize}


Provided  the inner product of any two columns and any two rows of $H$ is at most $1$,  $H$ satisfies the RC-constraint and can be taken as a parity-check matrix for a $(a^2,4,a)$-regular binary LDPC code.

In Section \ref{DCA} we show that for all odd $a>3$, difference matrices can be used to construct parity-check matrices (as described in Equation \eqref{eq:H0}) and hence codes  satisfying the RC-constraint. The specifications of these parity-check matrices, in terms of circulant submatrices, results in reduced storage requirements. Further, since $w_r=w\ll a^2$ and $w_c=4\ll 4a$ these matrices are sparse, leading to reduced decoding complexity. 


The removal of any of the $(4a)\times a$ submatrices of $H$ does not affect the RC-constraint, thus for any $\rho\in \{0,\dots, a-1\}$  we may define ${\overline H}$ to be a $(4a-1)\times (a^2-a)$ matrix of the form
\begin{small}
\begin{eqnarray}
{\overline H}=\hspace{-0.1cm}\left[\hspace{-0.1cm}
\begin{array}{rrrrr}
 R_0& R_1\dots&R_{\rho-1}&R_{\rho+1}\dots& R_{a-1}\\
 \mathfrak{P}_{0,0}&\mathfrak{P}_{0,1}\dots&\mathfrak{P}_{0,\rho-1} &\mathfrak{P}_{0,\rho+1}\dots&\mathfrak{P}_{0,a-1}\\
 \mathfrak{P}_{1,0}&\mathfrak{P}_{1,1}\dots&\mathfrak{P}_{1,\rho-1} &\mathfrak{P}_{1,\rho+1}\dots&\mathfrak{P}_{1,a-1}\\
 \mathfrak{P}_{2,0}&\mathfrak{P}_{2,1}\dots&\mathfrak{P}_{2,\rho-1} &\mathfrak{P}_{2,\rho+1}\dots&\mathfrak{P}_{2,a-1}\\
 \end{array}\right]\hspace{-0.2cm}
 \label{eq:H-rho}
\end{eqnarray}
\end{small}
where the row $\rho$ of all zeros is deleted.
Then under the assumption  that the inner product of any two columns and any two rows of ${\overline H}$ is at most $1$,  ${\overline H}$ can be taken as a parity-check matrix for a $(a^2-a,a,\{a-1,a\})$-near regular binary LDPC code that satisfies the RC-constraint.

The  parity-check matrix ${\overline H}_{12}$  given in Example \ref{ex-PE}
 provides an example of a parity-check matrix ($w_c=4$ and $w_r\in\{3,4\}$)  constructed in this manner.  Figure \ref{fig:H} provides an illustration of the general form of such matrices and is an example of a parity-check matrix for a $(650,4, \{25,26\})$-near regular binary LDPC code.
\begin{figure}
	\centering
	\includegraphics[width=\linewidth]{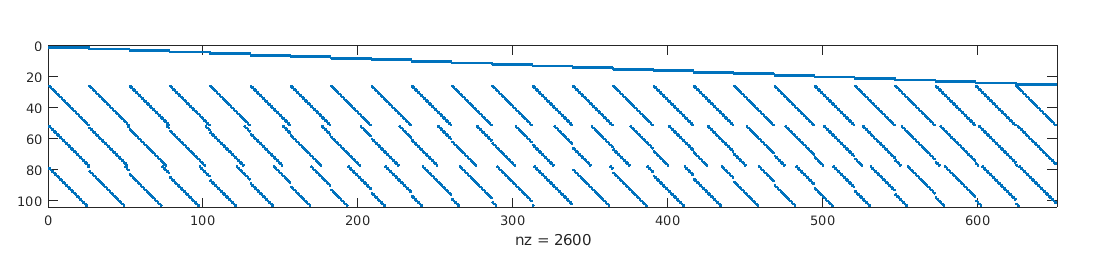}
	\caption{The parity-check matrix of the proposed  $(650,4, \{25,26\})$-near regular binary LDPC code.}
	\label{fig:H}
\end{figure}
It will also be shown in Section \ref{DCA} that for all even $a>2$, difference covering arrays can be used to construct parity-check matrices (as described in Equation \eqref{eq:H-rho})  and  codes satisfying the RC-constraint. As before, the specification of sparse parity-check matrices will result in reduced storage requirements and reduced decoding complexity.

Before we give these constructions it is useful to note that parity-check matrices can be visualised as graphs, with the rows of the parity-check matrix  associated with a set, $C=\{c_0,c_1,\dots, c_{m-\kappa-1}\}$, of  {\em parity-check nodes} and columns with a set $B=\{b_0,b_1,\dots,b_{m-1}\}$, of bits or {\em variable nodes}.  Then, the parity-check matrix $H=[H(i,j)]$ gives the {\em Tanner  graph}, $G(H)$, with vertex set  $C\cup B$
and an edge from $c_i\in C$ to $b_j\in B$ if and only if $H(i,j)=1$. 
As stated in \cite{Gallager63} and recently in \cite{Sarvaghad2020}, the bit error performance (BER) of  LDPC decoding, using the Sum-Product Algorithm (SPA), is affected by cycles of short length in the Tanner graph.
It can be shown that a parity-check matrix $H$ satisfies the RC-constraint if and only if all cycles in the Tanner  graph have length greater than 4, implying that the girth of the Tanner graph is at least 6, see \cite{Zhang10}.

Another factor  effecting the performance of a code is its stopping distance. A {\em stopping set}, $S$, is a subset  of  the set of variable nodes $B$ in $G(H)$, such  that  all  neighbors  of vertices in $S$ are adjacent to at least two vertices of $S$. In terms of the parity-check matrix $H$, a  stopping set $S$ of size $\ell$  is a subset of the columns of $H$ satisfying the property that the induced $(m-\kappa) \times \ell$ submatrix $H$  has  row sum $0$ or at least $2$, for all $m-\kappa$ rows.

 The existence of small stopping sets can adversely affect the performance of an LDPC code, with decoding failure caused when certain variable nodes are affected by errors after transmission. Thus the existence of small stopping sets can greatly  reduce  a code's error correcting capability. Stopping  sets  were  first  described  in  2002  by  Di  et. al.  \cite{Di}, when they were researching the average erasure probabilities of bits and blocks over a binary erasure channel (BEC). See \cite{Di,DPT,Gruner2013} for more details on  stopping sets. Let $\mathbb{S}$ denote  the  collection  of  all stopping sets in a Tanner graph, $G(H)$. Define the {\em stopping distance}, $s^*$, of $G(H)$ as the size of the smallest, non-empty stopping set in $\mathbb{S}$. It is known that the stopping distance of a code aids in the analysis of the code's error floor (an abrupt change in error rate curves arising from iterative decoding) and that the performance of an LDPC code over the BEC is dominated by the small stopping sets in the Tanner graph \cite{Richardson}. The larger the stopping distance, the lower the error floor  of  the  code. Also if a set of columns of the parity-check matrix is linearly dependent, then the corresponding vertices in the Tanner graph should have even degree in the induced subgraph. Thus the stopping distance provides a lower bound for the minimum distance of the code.

We use difference covering arrays (DCA) and difference matrices (DM) (as defined in Section \ref{DCA}) for our constructions of the parity-check matrices. These arrays can also be used to construct orthogonal Latin squares and nearly-orthogonal Latin squares (See \cite{Handbook} for related definitions). Article \cite{Gruner2013} lays the framework for using a full set of orthogonal Latin squares (equivalently transversal designs) to construct parity-check matrices for binary-LDPC codes. However, this analysis considers only orthogonal Latin squares that are constructed using finite fields, which exist only for a power of a prime. In \cite{Ferdosi} authors calculate the stopping distance of SA-LDPC codes  constructed by inflating  transversal designs of prime order, hence only giving codes of length a power of a prime.

In the current paper we significantly extend this work by generalizing the ideas to obtain quasi-cyclic-like codes for all orders even and odd. Together with the simulation results for the performance, we  prove some tight lower bounds for the stopping distance and minimum distance of the constructed codes. We show that all stopping sets of the constructed codes are of size at least $8$ when $a$ is not divisible by $3$. Furthermore, we present examples where our simulation results illustrate that the stopping distance is $10$. More importantly we  analyze the minimum distance of the codes  constructed here and prove that a large infinite family (more specifically when the smallest prime dividing $a$ is greater than 5) of these codes have minimum distance $10$ and are quasi-cyclic in structure.

On the other hand,  our construction is closely related to other well-known constructions in the literature, that focus on finding codes with large girth. If the $R_i$'s (defined in Equation \ref{eq:H0}) are removed the resulting parity check matrix  will be equivalent to those in QC-LDPC codes that are cyclic liftings of fully connected base graphs of type $(3,a)$ with minimum lifting factor for girth $6$. These codes are classified for small $a$ in \cite{Tasdighi}. See \cite{Ranganathan}, \cite{DonovanRaoYazici2017} and \cite{Fossorier} for related work and definitions. Article \cite{Amirzade} classifies codes that are liftings of fully connected base graphs of type $(4,a)$ with minimum lifting factor for girth $6$ and $5\leq a\leq 11$. When the smallest prime dividing $a$ is greater than $5$, the infinite family of quasi-cyclic codes constructed in this paper have girth $6$,  minimum distance $10$ and are examples of codes constructed by liftings of fully connected base graphs of type $(4,a)$ with lifting degree $a$ which is the smallest possible lifting degree.

\section{Difference Matrices and Difference Covering Arrays} \label{DCA}

The parity-check matrices for regular and near regular binary LDPC codes are constructed using difference matrices and difference covering arrays, respectively.

A {\em difference matrix}, DM$(k;a)$,
is defined to be a $a\times k$ array  $D=[D(i,j)]$, where
\begin{itemize}
\item[-] all entries in the first column of $D$ are $0$ and all remaining columns contain each entry $0,\dots, a-1$ precisely once, and
\item[-]  for all pairs of distinct columns, $j$ and $j^\prime$,
the differences $D(i,j)-D(i,j^\prime)$  \mbox{ mod }$a$, for $0\leq i\leq a-1$, are distinct; that is, $\{D(i,j)-D(i,j^\prime)$ \mbox{ mod }$a$ $\mid 0\leq i\leq a-1\}=\{0,\dots,a-1\}$.
\end{itemize}
Difference matrices are well studied in the literature, see \cite{Handbook} for
 constructions. It can be shown that difference matrices with more than 2 columns do not exist
 for even orders, but that DM$(3;a)$ difference matrices exist for all odd $a$. (See \cite{Drake} and \cite{Handbook} Section VI.17). Further, for positive integer $n$, a DM$(3;2n+1)$ corresponds to an additive permutation, with the numbers of distinct DM$(3;2n+1)$ corresponding to the sequence A002047 in Sloane's encyclopedia \cite{Sloane} and current enumerations giving  the number of distinct additive permutations  for $2n+1=23$ (distinct DM$(3;23)$)  as $577,386,122,880$.

 It is clear from the definition, that permuting rows does not change the underlying properties of a
   difference matrix. Hence we will assume that all difference
   matrices are in the standard form, namely  $D(i,1)=i$ for all rows $0\leq i\leq a-1$.
\begin{example} \label{ex:47} $D_5$  is an example of a  DM$(3;5)$, whereas
$D_7$ is a DM$(4;7)$. Notice that the  property above is
 satisfied, as  for instance, the set of differences between the second and third
  rows in $D_5^T$  (the transpose of $D_5$), is $\{0,1,2,-2,-1\}$ which equals $\{0,1,2,3,4\}$
  when working modulo $5$.
\begin{eqnarray*}{ll}\label{eq:47}
\begin{array}{ll}
D_5^T=\left[\begin{array}{ccccc}
0&0&0&0&0\\
0&1&2&3&4\\
0&2&4&1&3\\
\end{array}\right], &
D_7^T=\left[\begin{array}{ccccccc}
0&0&0&0&0&0&0\\
0&1&2&3&4&5&6\\
0&2&4&6&1&3&5\\
0&3&6&2&5&1&4\\
\end{array}\right]
\end{array}
\end{eqnarray*}
\end{example}

 Let $I$ denote the $a\times a$ identity matrix, and $P^i$ be a circulant permutation matrix (CPM) obtained from $I$ by cyclically shifting its rows $i$ positions to the left. We set $P^1=P$ and $P^0=I$. Note that row $r$ of $P^i$ has precisely one entry equal to $1$ at column $r-i\ mod\ a$ and all other entries equal to $0$.
In Construction 1 below, it is demonstrated that, for all odd integers $a$, a DM$(3;a)$ can be used to construct a matrix $H=[H(i,j)]$, of the form given in Equation \eqref{eq:H0}, that is a parity check matrix for a $(a^2,4,a)$-regular binary LDPC code.

\vspace{0.3cm}
\noindent{\bf CONSTRUCTION 1} Let $a\geq 3$ be an odd positive integer and
 $D=[D(i,j)]$ be a DM$(3;a)$ in standard form.
Construct $H$ as given in Equation \eqref{eq:H0} where for $u=0,1,2$ and $v=0,\dots, a-1$,
  the matrices  $R_v=[R_v(i,j)]$ and $\mathfrak{P}_{u,v}$ satisfy
  \begin{eqnarray}
\begin{array}{ll}
R_v(i,j)=1, \mbox{for } i=v \mbox{ and } j=0,\dots,a-1,&  \\
R_v(i,j)=0,  \mbox{otherwise},&\label{eq:HDMR}\\
\mathfrak{P}_{0,v}=P^{D(v,0)}=P^0=I,& \nonumber\\
\mathfrak{P}_{1,v}=P^{D(v,1)}=P^v,& \nonumber\\
\mathfrak{P}_{2,v}=P^{D(v,2)}.& \label{eq:HDMP}
\end{array}
\end{eqnarray}
We then have
\begin{eqnarray}
 H=\left[
 \begin{array}{cccc}
 R_0& R_1& \dots& R_{a-1}\\
 I&I&\dots &I\\
 I&P^1&\dots &P^{a-1}\\
 P^{D(0,2)}&P^{D(1,2)}&\dots &P^{D(a-1,2)}\\
 \end{array}\right].\label{eq:H}
 \end{eqnarray}
The simplicity of this construction and the fact that the underlying combinatorial structure is the cyclic group of odd order, and thus, the binary operation is addition modulo $a$, allows us to verify that the RC-constraint is satisfied as well as analytically determining bounds for the minimum distance of the code, the rate of the code and the size of the minimum stopping set, as shown below.
In addition, it is only necessary to store the DM$(3;a)$. The entries in row $v$ of this array DM$(3;a)$ then determine the non-zero entries in the first column of each of the $a\times a$ submatrices $R_v$ and $\mathfrak{P}_{u,v}$, for $v=0,\dots, a-1$, (more precisely the non-zero entries of $\mathfrak{P}_{1,v}$ and $\mathfrak{P}_{2,v}$), with all remaining columns of $\mathfrak{P}_{1,v}$ and $\mathfrak{P}_{2,v}$  taken as cyclic shifts of the first column.

Furthermore, for $a\geq 5$ we give a family of DM$(3;a)$,   resulting in a parity-check matrix in quasi-cyclic form after row and column permutations. These features  greatly enhance applicability of the resulting $(a^2,4,a)$-regular binary LDPC code.

In addition, Construction 1 can be generalised  and the existence of a DM$(k;a)$ used to construct a $(a^2,k+1,a)$-regular binary LDPC code, but the existence of such DM$(k;a)$'s is not known for all admissible $a$. In particular, as stated above, difference matrices do not exit for even order $a$.  However, for even order we are able to adapt the above construction using the next best structure, namely difference covering arrays where we cover as
  many differences as possible.
  It is this adaption that forms one of the main innovations of this paper, as it demonstrates  the adaption of the above basic construction  to support variability in code lengths and rates. In the second construction,  the resulting parity-check matrix takes a similar form, namely the form given in Equation \eqref{eq:H-rho} where ${\overline H}$ is similar to $H$ (Equation \eqref{eq:H}) except that a $(4a)\times a$ submatrix has been removed, as well as a row of all zeros.

We start with verifying the properties necessary to show that, when $a$ is odd, the $(a^2,4,a)$-regular binary LDPC code satisfies the RC-constraint. This argument is then extended to  $a$ even, and Construction 2 (page \pageref{Construct2}) used to obtain
  a $(a^2-a,4,\{a-1,a\})$-near regular binary LDPC code that satisfies the RC-constraint.

 In what follows let $a\geq 3$ be odd and $D=[D( i,j)]$ be a DM$(3;a)$ in standard form. Take $H$ to be a  $(0,1)$ matrix constructed as in Construction 1 using a DM$(3;a)$. First, we give some straightforward  observations  useful for later proofs.
\begin{lemma}\label{lem:w_c=4}
For any $b\in \{0,\dots, a^2-1\}$, column $b$ has sum $4$. Further, given $x<y<z<t$ such that $H(x,b)=H(y,b)=H(z,b)=H(t,b)=1$, then $0\leq x\leq a-1$, $a\leq y\leq 2a-1$, $2a\leq z\leq 3a-1$ and $3a\leq t\leq 4a-1,$  with
$y=q\mbox{ mod }a$, $z=q+D(x,1)\mbox{ mod }a=q+x\mbox{ mod }a$ and $t=q+D(x,2)\mbox{ mod }a$ for some $0\leq q\leq a-1$.
\end{lemma}
Below is some notation that will be used in later proofs.
The rows of $H$ are partitioned into four subsets denoted $V_{-1}=\{0,\dots,a-1\}$, $V_{0}=\{a,\dots,2a-1\}$, $V_{1}=\{2a,\dots,3a-1\}$ and
$V_{2}=\{3a,\dots,4a-1\}$, that is, respectively, in to the rows of $R_v$, and $\mathfrak{P}_{u,v}$, $u=0,1,2$. Further for each  column \\$b\in \{0,\dots, a^2-1\}$, define
\begin{eqnarray}
\begin{array}{l}
C_b=\{x,y,z,t\mid H(x,b)=H(y,b)=H(z,b)=H(t,b)=1\},\label{eq:Cb}
\end{array}
\end{eqnarray}
that is, $C_b$ gives the set of rows of $H$ with entry $1$ in column $b$.
Consequently,\begin{eqnarray}
&C_{xa+q}=\{x,y,z,t\}=\{x,q+a,(q+D(x,1)\mbox{ mod }a)+2a,\nonumber\\
&(q+D(x,2)\mbox{ mod }a)+3a\}. \label{eq:xyzt}
\end{eqnarray}
for $0\leq x,q\leq a-1$.
\begin{lemma}\label{lm:RPIproperties}
The inner product of any two rows of $H$ is at most one.
\end{lemma}
\begin{proof}
Since $\mathfrak{P}_{u,v}$ is a permutation of the identity matrix for $u=0,1,2$ and $v=0,\dots, a-1$, it follows immediately that
\begin{itemize}
\item[(a)]\label{InProA} The inner product of any two rows of  $R_v$ or $\mathfrak{P}_{u,v}$ is zero. Hence, the inner product of any two rows of $H$ in the same subset $V_i$ is $0$ for all $-1\leq i\leq 2$.
\item[(b)]\label{InProB} The inner product of any row of $\mathfrak{P}_{u,v}$ with any row of $R_v$ is at most one. Furthermore, for any row $r$ in $V_{-1}$, there is precisely one $v$ such that row $r$ of $R_v$ is not the zero vector. Hence, the inner product of any row of $V_{-1}$ with any row of $V_i$ is exactly one for $0\leq i\leq 2$.
\item[(c)] \label{InProC}
Finally, assume that there exists  rows $r$ and $r^\prime$ and distinct $v$ and $v^\prime$ such that the inner product of rows $r$ and $r^\prime$ of $\mathfrak{P}_{i,v}$ and $\mathfrak{P}_{j,v}$ is equal to one as is the inner product of rows
$r$ and $r^\prime$ of $\mathfrak{P}_{i,v^\prime}$ and $\mathfrak{P}_{j,v^\prime}$ for some $0\leq i<j\leq 2$. This implies that \begin{eqnarray*}
r-D(v,i)\mbox{ mod }a&=&r^\prime-D(v,j)\mbox{ mod }a\mbox{ and}\\
r-D(v^\prime,i)\mbox{ mod }a&=&r^\prime-D(v^\prime,j)\mbox{ mod }a.
\end{eqnarray*}
Rearranging the difference of these equations gives \begin{eqnarray*}
D(v,i)-D(v,j)&=&D(v^\prime,i)-D(v^\prime,j)\mbox{ mod }a,\end{eqnarray*}
which contradicts the definition of a DM$(3;a)$. Hence the inner product of any two rows $r\in V_i$ and $r^\prime\in V_j$, for $0\leq i< j\leq 2$, is at most one. On the other hand $r$ and $r^\prime$  both contain one in the same column when $r^\prime=r+D(v,i)-D(v,j)\ mod\ a$. Therefore the inner product of any two rows $r\in V_i$ and $r^\prime\in V_j$, for $0\leq i< j\leq 2$, is exactly $1$.
\end{itemize}
\end{proof}
\begin{lemma}\label{lem:H} Let $a\geq 3$ be odd.
  Then $H$ is a parity-check matrix for a $(a^2,4,a)$-regular binary LDPC code that satisfies the  RC-constraint.
\end{lemma}
\begin{proof}
The fact that the inner product of any two rows of $H$ is at most one follows from Lemma \ref{lm:RPIproperties}.

Now consider any two distinct columns $b=va+q$ and $b^\prime =v^\prime a+q^\prime$, where $0\leq v,v^\prime, q,q^\prime\leq a-1$ and $b^\prime> b$, of the parity-check matrix.

By definition the set of rows that contain $1$ in the columns $b$ and $b^\prime$  is given by $C_b$ and $C_{b^\prime}$ respectively. The inner product of column $b$ and $b^\prime$ is given by $\mid\hspace{-0.1cm} C_b\cap C_{b^\prime}\hspace{-0.1cm}\mid$. By Equation \eqref{eq:xyzt},
$C_b=C_{av+q}=
\{v,q+a,(q+D(v,1)\mbox{ mod }a)+2a,
(q+D(v,2)\mbox{ mod }a)+3a\}$ and
$C_{b^\prime}=C_{av^\prime+q^\prime}=
\{v^\prime,q^\prime+a,(q^\prime+(D(v^\prime,1)\mbox{ mod }a)+2a,
(q^\prime+D(v^\prime,2)\mbox{ mod }a)+3a\}$

We then have the following cases:

Case 1: If $v=v^\prime$, since $b$ and $b^\prime$ are distinct columns,  $q\neq q^\prime$. We have:
\begin{eqnarray*}
q+a&\neq& q^\prime+a,\\
 q+D(v,1)&\neq& q^\prime+D(v^\prime,1)\mbox{ mod }a \mbox{ and }\\
 q+D(v,2)&\neq& q^\prime+D(v^\prime,2)\mbox{ mod }a.
 \end{eqnarray*}
Hence the inner product of column $b$ and $b^\prime$ is $1$.

Case 2: If $v\neq v^\prime$ then $D(v,1)\neq D(v^\prime,1)$ and $D(v,2)\neq D(v^\prime,2)$.

 Case 2.1 If $q=q^\prime$ then $a+q=a+q^\prime$, $q+D(v,1)\neq q^\prime+D(v^\prime,1)$ and $q+D(v,2)\neq q^\prime+D(v^\prime,2)\mbox{ mod }a$ and the inner product of column $b$ and $b^\prime$ is $1$.

Case 2.2 If $q\neq q^\prime$ then $q+a\neq q^\prime+a$. Now assume
\begin{eqnarray*}
(q+D(v,1)\mbox{ mod }a)+2a&=& (q^\prime+D(v^\prime,1)\mbox{ mod }a)+2a,\\
(q+D(v,2)\mbox{ mod }a)+3a&=&(q^\prime+D(v^\prime,2)\mbox{ mod }a)+3a.
\end{eqnarray*}

This implies $q-q^\prime=D(v^\prime,1)-D(v,1)=D(v^\prime,2)-D(v,2)$. Hence $D(v,1)-D(v,2)=D(v^\prime,1)-D(v^\prime,2)$ which contradicts the definition of a $DM(3;a)$. Therefore the inner product of column $b$ and $b^\prime$ is at most $1$.

Thus the inner product of any pair of columns of  $H$, is at most one, as required. Hence $H$ satisfies the RC-constraint.
\end{proof}

Let $n$ be a positive integer. We now present a second construction replacing the DM$(3;2n+1)$ with a difference covering array DCA$(3;2n)$ that exists for all $n\geq 2$.

Define a {\em difference covering array}, DCA$(k;2n)$,  to be a $2n\times k$ array
 $D=[D(i,j)]$, where
\begin{itemize}
\item[-]  all entries in the first column of $D$ are $0$ and the remaining columns
contain each  entry $0,1,\dots, 2n-1$  precisely once, and
\item[-]  for all pairs of distinct non-zero columns, $j$ and $j^\prime$,  the differences $D(i,j)-D(i,j^\prime)$  \mbox{ mod }$2n$, for $0\leq i\leq 2n-1$, are non-zero and cover the set $\{1,2,\dots,2n-1\}$.
\end{itemize}

Similar to difference matrices we will assume that all difference covering arrays
studied here are in standard form with  $D(i,1)=i$, for all $0\leq i\leq 2n-1$.
\begin{example} \label{ex46} $D_4$  is an example of a DCA$(3;4)$, whereas
  $D_6$ is a DCA$(4;6)$. Notice that the first property above is
   satisfied and, for instance, the set of differences between the second and third
    rows in $D_4^T$  (the transpose of $D_4$) is $\{-1,-2,2,1\}$ which equals
    $\{3,2,2,1\}$ when working modulo $4$.
\begin{eqnarray}
\begin{array}{ll}
D_4^T=\left[\begin{array}{ccccc}
0&0&0&0\\
0&1&2&3\\
1&3&0&2\\
\end{array}\right]&
D_6^T=\left[\begin{array}{ccccccc}
0&0&0&0&0&0\\
0&1&2&3&4&5\\
1&3&5&0&2&4\\
3&0&4&1&5&2\\
\end{array}\right]
\end{array}
\end{eqnarray}
\end{example}
In a difference covering array there are $2n$ rows and, since the
 $2n$ differences $D(i,j)-D(i,j^\prime)$ are non-zero, it can be shown that for any pair of distinct non-zero
 columns $j$ and $j^\prime$ there exists
 two rows $r_0$ and $r_1$ such that $D(r_0,j)-D(r_0,j^\prime)=n= D(r_1,j)-D(r_1,j^\prime)$.
  That is, the repeated difference is  $n$ (see \cite{DDHKR} for a proof).
While not a lot is known for general $k$,  when $k=3$ it is known that as $n$ grows the
 number of distinct  DCA$(3;2n)$ grows significantly, see \cite{DDKM}. For further results on
  difference covering arrays see \cite{Yin1} and \cite{Yin2}.
  
\vspace{0.3cm}
\noindent{\bf CONSTRUCTION 2} \label{Construct2} Let $a\geq 4$ be an even positive integer,
  $D=[D(i,j)]$ be a DCA$(3;a)$, where $r_0$ represents precisely one of the two rows where $D(r_0,2)-D(r_0,1)=n$.
  Then  construct ${\overline H}$ as given in
  Equation \eqref{eq:H-rho}  where, for $u=0,1,2$,
  $v=0,\dots, a-1$ and  $v\neq r_0$, the matrices  $R_v=[R_v(i,j)]$ and
  $\mathfrak{P}_{u,v}$ are as given below and  the row $r_0$ of all zeros has been removed.
\begin{eqnarray}
\begin{array}{ll}
R_v(i,j)=1, \mbox{for } i=v \mbox{ and } j=0,\dots,a-1,\nonumber \\
R_v(i,j)=0,  \mbox{otherwise},\\ \label{eq:HDCAR}
\mathfrak{P}_{0,v}=P^{D(v,0)}=P^0=I,& \nonumber\\
\mathfrak{P}_{1,v}=P^{D(v,1)}=P^v,& \nonumber\\
\mathfrak{P}_{2,v}=P^{D(v,2)}& \label{eq:HDCAP}
\end{array}
\end{eqnarray}
\begin{example}\label{pbibd4}
Take $D_4$ as set out in Example \ref{ex46} with $a=4$ and $r_0=2$. Then Construction 2 gives the parity-check matrix displayed in Example \ref{ex-PE}.
\end{example}
\begin{lemma}\label{lem:H-rho} Let $a\geq 4$ be even, then ${\overline H}$ given in Construction 2 is a parity-check matrix for a $(a^2-a,4,\{a-1,a\})$-near regular binary LDPC code that satisfies the  RC-constraint.
\end{lemma}
\begin{proof}
The proof, for $a$ even, follows as in the proof for Lemma \ref{lem:H} ($a$ odd) where the row and column ranges have been relabeled appropriately.
\end{proof}
We now examine the properties of the above codes and get exact bounds for the rate of the code as well as minimum distance and the size of the smallest stopping sets.

From here on we assume that

--For $a$ even, $D=[D(i,j)]$ is a DCA$(3;a)$ and ${\overline H}$, as given in Construction 2 on $D=[D(i,j)],$ is the parity-check matrix of a
$(a^2-a,4,\{a-1,a\})$-near regular binary LDPC code.

--For $a$ odd, $D=[D(i,j)]$ is a DM$(3;a)$ and $H$ as given in Construction 1 on $D=[D(i,j)]$ is the parity-check matrix of a $(a^2,4,a)$-regular binary LDPC code.

\section{Properties of LDPC Codes from DM$(3;2n+1)$ and DCA$(3;2n)$}\label{properties}
In this section we prove that the LDPC codes constructed in this paper have no cycles of length less than 6 in the associated Tanner graph. We study the rank and the stopping distance of these codes, determine the minimum distance of a particular class of these codes, and also give an  algebraic proof that this particular class of codes are quasi-cyclic in structure.

We show that the rank of the parity-check matrices $H$ and ${\overline H}$  respectively, equals $4a-3$ for odd $a$ and  less than or equal to $4a-6$ for even $a$. (Note that since we are working with $(0,1)$ matrices the linear dependence of rows is calculated bitwise over the binary field ${\Bbb Z}_2$.)
Then, recalling that for any  matrix $A$  with $m$ columns the $\mbox{rank}(A)+\mbox{nullity}(A)=m$, for $a$ odd, the dimension  of the code with parity check matrix $H$ equals $a^2-4a+3$,  and, for $a$ even,   the dimension of the code with parity check matrix ${\overline H}$ is greater than  or equal to $a^2-5a+6.$ Thus we will verify that the rate of the $(a^2,4,a)$-regular binary LDPC code in Construction 1 is equal to Identity \eqref{eq:rankH} (below) while the rate of the $(a^2-a,4,\{a-1,a\})$-near regular binary LDPC code in Construction 2 is greater than or equal to Identity \eqref{eq:rankH-rho}:
\begin{eqnarray}
Rate&=&1-\frac{(4a-3)}{a^2},\mbox{ for } a\mbox{ odd } (H),\mbox{ and}\label{eq:rankH}\\
Rate&=&1-\frac{(4a-6)}{a^2-a}, \mbox{ for } a\mbox{ even } ({\overline H}).\label{eq:rankH-rho}
\end{eqnarray}
The rate has been enumerated in Table \ref{rate-table} for $a=12$ to $30$. It can be seen how the rate of the code increases as $m$ increases. Indeed it can be shown algebraically that the rank tends to $1$ as $a$ tends to infinity.
\begin{table}[ht]
\caption{The rates of some LDPC codes constructed in this paper for $a \geq 12$.}\label{rate-table}
\begin{center}
\begin{tabular}{|l|p{2cm}|p{2.5cm}|p{2cm}|}
\hline
$a$ odd&Code Length $m$ & Code Dimension& Code Rate \\
\hline
13 & 169 & 120 & 0.71\\
15 & 225 & 168 & 0.75 \\
17 & 289 & 224 & 0.78\\
19& 361& 288 & 0.80\\
21 & 441 & 360 & 0.82\\
23& 529 & 440 & 0.83\\
25 & 625 & 528 & 0.84 \\
27 & 729 & 624 & 0.86 \\
29 & 841 & 728 & 0.87  \\
39 & 1521 & 1368 & 0.90\\
\hline
$a$ even&Code Length $m$ & Code Dimension Lower Bound & Code Rate Lower Bound    \\\hline
12 & 132 & 90 & 0.68\\
14 & 182 & 132 & 0.72 \\
16 & 240 & 182 & 0.76\\
18& 306& 240 & 0.78\\
20 & 380 & 306 & 0.81\\
22& 462 & 380 & 0.82\\
24 & 552 & 462 & 0.84 \\
26 & 650 & 552 & 0.85 \\
28 & 756 & 650 & 0.86  \\
30 & 870 & 756 & 0.87\\
\hline
\end{tabular}
\end{center}
\end{table}

We start by noting that  Lemmata \ref{lem:H}  and \ref{lem:H-rho} verify that for both $H$ and ${\overline H}$ the inner product of any two columns is less than or equal to one giving the following bound on the girth of the Tanner graph.
\begin{lemma}\label{girth}
The Tanner graph of the constructed $(a^2,4,a)$-regular and the $(a^2-a,4,\{a,a-1\})$-near regular binary LDPC codes have girth at least 6.
\end{lemma}
In Lemma \ref{lem:rankH} we show  that, for $a$ odd, the matrix $H$ has rank $4a-3$. Then, in Lemma \ref{lem:rankH-rho}, we will show that, for $a$ even, the matrix ${\overline H}$ has rank  at most $4a - 6$.
\begin{lemma}\label{lem:rankH} Let $a\geq 3$ be odd, then 
the rank of the matrix $H$  is exactly $4a - 3$.
\end{lemma}
\begin{proof}
We will establish this result by showing that there is a set of $3$ rows in the row space of $H$ such that each row may be written as linear combinations of the remaining $4a-3$ rows. Further we will establish that there exists a set of  $4a-3$ rows that are linearly independent.

Lemma \ref{lem:w_c=4}  implies that, for each  column $b$, there exists $x\in V_{-1}$, $y\in V_{0}$,  $z\in V_{1}$ and $t\in V_{2}$ with $1=H(x,b)=H(y,b) =H(z,b)=H(t,b)$, where $V_i=\{(i+1)a,\dots,(i+2)a-1\}$, $i=-1,0,1,2$.
Thus when the row space of $H$ is restricted to the rows given by any two of these sets, $V_{i,j}=V_i\cup V_j$ where $-1\leq i<j\leq 2$, the bitwise sum of the corresponding rows is $0$ modulo $2$, implying the rows corresponding to $V_{i,j}$, $-1\leq i<j\leq 2$, are linearly dependent.
Thus at least 1 row from 3 distinct $V_i$'s  needs to be removed to obtain a subset of linearly independent rows of $H$, or equivalently, the size of any linearly independent subset is at most $4a-3$. Without loss of generality  we will eliminate the rows $2a-1$, $3a-1$ and $4a-1$.

  Now we claim that the set of rows corresponding to $V=V_{-1}\cup V_0 \cup V_1 \cup V_2 \setminus \{2a-1,3a-1,4a-1\}$ is a linearly independent set.

  Assume that this is not the case and that there exists $L\subseteq V$ such that the corresponding rows of $H$ give  a linearly dependent set. Note that, this implies when $H$ is restricted to the rows of $L$ then the sum of the entries in columns of $H$ is $0\ mod\ 2$

  For any $i\in \{-1,0,1,2\}$ and any $r\in V_i$, the row sum of row $r$ is $a$, so let  $\{b_{r_1},\dots, b_{r_a}\}$ denote the set of $a$ columns where $H(r,b_{r_j})=1$ for $j=1,\dots, a$. The proof of Lemma \ref{lm:RPIproperties} implies that for distinct  $r,r^\prime\in V_i$ if $H(r,b_{r_j})=1$, then $H(r^\prime,b_{r_j})\neq 1$. Further, for distinct $i,i^\prime$ and $r\in V_i$ and $r^\prime \in V_{i^\prime}$, there exists a unique $b_{r_j}\in \{b_{r_1},\dots, b_{r_a}\}$ such that $H(r,b_{r_j})=1$ and $H(r^\prime,b_{r_j})=1$. 

Next we proceed by  assuming $r\in L\cap V_i$, for some $i\in \{-1,0,1,2\}$ and let $H(L,\{b_{r_1},\dots, b_{r_a}\})$ denote the restriction of the matrix $H$ to rows in $L$ and columns $b_{r_1}, b_{r_2}, \dots, b_{r_a}$.
    Since $L$ corresponds to  a linearly dependent set, summing over all rows of $H(L,\{b_{r_1},\dots, b_{r_a}\}) \mod 2$ gives the $0$ vector of length $a$. Hence the number of non-zero entries in $H(L,\{b_{r_1},\dots, b_{r_a}\})$  is even, say $2\ell$ for some $\ell\in \mathbb{Z}$. But now the argument above implies that $2\ell=a+\mid L\setminus V_i\mid$ (the number of rows in $L$ but not in  $V_i$). Now as $a$ is assumed to be odd, we have $\mid L\setminus V_i\mid$ is also odd.

  Next assume $L\cap V_i\neq \emptyset$ and $r^\prime \in V_i$ but $r^\prime\notin L$. Again let $\{b_{r^\prime 1},\dots, b_{r^\prime a}\}$ denote the set of columns such that $H(r^\prime,b_{r^\prime _ j})=1$, $j=1,\dots, a$. As above let $H(L,\{b_{r^\prime _1},\dots, b_{r^\prime _a}\})$ be the restriction of the matrix $H$ to rows in $L$ and columns in $\{b_{r^\prime _1},\dots, b_{r^\prime _a}\}$.  Again the number of entries in $H(L,\{b_{r^\prime _1},\dots, b_{r^\prime _a}\})$ is even, say $2\ell^\prime,$ for some $\ell^\prime\in \mathbb{Z}$. We have $2\ell^\prime=\mid L\setminus V_i\mid$, implying  $\mid L\setminus V_i\mid$ is even. Thus we have a contradiction, and no such $r$ exists.

  So if $L\cap V_i\neq \emptyset$ then $V_i\subseteq L\subseteq V=V_{-1}\cup V_0 \cup V_1 \cup V_2 \setminus \{2a-1,3a-1,4a-1\}$ implying $L=V_{-1}$. But  the sum of each column restricted to $V_{-1}$ is 1, again giving a contradiction.

Hence $V$ is linearly independent and the rank of $H$ is at least $4a-3$ for odd $a$, implying the rank of $H$ is  \emph{exactly} $4a-3$ for odd $a$.
\end{proof}
\begin{lemma}\label{lem:rankH-rho} Let $a\geq 4$ be even,  
then
the rank of the matrix ${\overline H}$  is at most $4a-6$.
\end{lemma}
\begin{proof}
In this proof we establish the bound on the rank by showing that there is a set of $5$ rows in the row space of ${\overline H}$ (i.e. the set of vectors corresponding to rows of ${\overline H}$) that can be written as linear combinations of the remaining rows. To aid understanding we will prove the result for a $(4a)\times (a^2-a)$  matrix ${\overline H}_{r_0}$ that agrees with ${\overline H}$ in rows $0$ to $r_0-1$, with the next row $r_0$ having all entries equal to $0$, and then followed by rows $r_0 + 1$ to $4a-2$ of ${\overline H}$. Thus we have reinstated the row of all zeros to ${\overline H}$. This will allow us to simplify the arguments, while the introduction of a zero row will not change the calculation of the rank of ${\overline H}$.

We proceed by using Construction 2 to deduce the following properties of ${\overline H}_{r_0}$.

Firstly, note that as in the case for odd $a$, when the row space of ${\overline H}_{r_0}$ is restricted to rows given by any two of the sets, $V_{i,j}=V_i\cup V_j$ where $-1\leq i<j\leq 2$, the componentwise sum of the corresponding vectors is $0$ modulo $2$, implying the vectors corresponding to $V_{i,j}$, $-1\leq i<j\leq 2$, are linearly dependent.

To obtain a subset of linearly independent vectors of the row space of ${\overline H}_{r_0}$ at least $4$ rows need to be removed, or equivalently, the size of any linearly independent subset is at most $4a-4$. Without loss of generality  we will eliminate the vectors corresponding to  rows $r_0$, $2a-1$, $3a-1$ and $4a-1$.

Now consider the following sets $L_3,\ L_4,\ L_5$ of rows of ${\overline H}_{r_0}$. For $0\leq i\leq \frac{a-1}{2}$
\begin{eqnarray*}
L_3&=&\{x, a+2i,2a+2i\ \mid\  D(x,1)\equiv1\mod 2\}\\
L_4&=&\{x,a+2i,3a+2i\ \mid\ D(x,2)\equiv1\mod 2 \}\\
L_5&=&\{x,2a+2i,3a+2i\ \mid\ D(x,1)\not\equiv D(x,2)\mod 2\}.
\end{eqnarray*}
We claim that the vectors corresponding to each of these sets are linearly dependent.

To see that each of $L_3$ and $L_4$  corresponds to a linearly dependent set of vectors observe that for any column $b$ and $C_b=\{x,y,z,t\}$, Lemma \ref{lem:w_c=4} implies if $D(x,1)\equiv 0\mod 2$ then $y\equiv z\mod 2$ and if $D(x,1)\equiv 1\mod 2$ then $y\not\equiv z\mod 2$. Similarly if $D(x,2)\equiv 0\mod 2$ then $y\equiv t\mod 2$ and if $D(x,2)\equiv 1\mod 2$ then $y\not\equiv t\mod 2$. Furthermore, to see that $L_5$  corresponds to a linearly dependent set of vectors observe that, if $D(x,1)\equiv D(x,2)\mod 2$ then $z\equiv t\mod 2$ and if $D(x,1)\not\equiv D(x,2)\mod 2$ then $z\not\equiv t\mod 2$.

 Note that there are no rows that are common to all three sets and  $L_3,L_4,L_5\subseteq V=V_{-1}\cup V_0 \cup V_1 \cup V_2 \setminus \{r_0,2a-1,3a-1,4a-1\}$. So as $L_3,L_4$ and $L_5$ are linearly dependent sets of rows then to obtain a subset of linearly independent vectors of the row space of ${\overline H}_{r_0}$,  without loss of generality we  can further eliminate the vectors corresponding to rows $2a-2$ and $3a-2$.

Hence we can eliminate the rows $r_0,2a-2,2a-1,3a-2,3a-1,4a-1$ without changing the rank of ${\overline H}_{r_0}$. So the rank of ${\overline H}$ is at most $4a-6$ for even $a$.
\end{proof}
We want to make a note here that our simulations show that the rank of $\overline{H}$ is exactly  $4a-6$ for all even $a\leq 200.$
\begin{lemma}
For $a$ odd, the $(a^2,4,a)$-regular binary LDPC code has rate $\displaystyle 1-\frac{4a-3}{a^2}$ and, for $a$ even,  the $(a^2-a,4,\{a-1,a\})$-near regular binary LDPC code has rate at least $\displaystyle 1-\frac{4a-6}{a^2-a}$.
\end{lemma}
\begin{proof}
Recall that the $\mbox{nullity}(H) = m-\mbox{rank}(H)$, and so, respectively,  Lemmata \ref{lem:rankH} and \ref{lem:rankH-rho} imply that for $a$ odd with $m = a^2$, and for $a$ even with  $m=a^2-a$,
 rate of the code  is  exactly $\displaystyle \frac{a^2-4a+3}{a^2}=1-\frac{4a-3}{a^2}$ and at least $\displaystyle \frac{a^2-5a+6}{a^2-a}=1-\frac{4a-6}{a^2-a}$ respectively.
\end{proof}
\subsection{Stopping and Minimum Distance of the LDPC codes from DM$(3;a)$ and DCA$(3;a)$}

To investigate possible stopping sets for our codes we exploit the facts that  the parity-check matrix $H$ has column weight $4$, satisfies the RC-constraint and that for each column $b_i$ we may define sets $C_{b_i}=\{x_i,y_i,z_i,t_i\}$ as in Equation \eqref{eq:xyzt} where $x_i\in X$, $y_i\in Y$, $z_i\in Z$ and $t_i\in T$ for disjoint sets $X,Y,Z,T$. Then an exhaustive computer search can be used to show that, up to isomorphism there are only two possible stopping sets $S_1$ and $S_2$ of size $6$ that can occur as subsets of columns of parity-check matrices with these properties, for arbitrary columns $b_1,\ b_2,\ b_3,\ b_4,\ b_5,\ b_6$. Note that, one can also prove analytically that there exists only 2 non-isomorphic cases by considering having only 2 different $x_i$'s (case $S_1$) and 3 different $x_i$'s (Case $S_2$). Then the $y_i$'s, $z_i$'s and $t_i$'s are uniquely determined up-to-isomorphism.$S_1$ and $S_2$ take the following forms.
\vspace{0.1cm}

\begin{eqnarray*}
\begin{array}{ll}
S_1=\{&C_{b_1}=\{x_1,y_1,z_1,t_1\},\\
&C_{b_2}=\{x_1,y_2,z_2,t_2\},\\
&C_{b_3}=\{x_1,y_3,z_3,t_3\},\\
&C_{b_4}=\{x_2,y_1,z_2,t_3\},\\
&C_{b_5}=\{x_2,y_2,z_3,t_1\},\\
&C_{b_6}=\{x_2,y_3,z_1,t_2\}\},
\end{array}&&
\begin{array}{ll}
S_2=\{&C_{b_1}=\{x_1,y_1,z_1,t_1\},\\
&C_{b_2}=\{x_1,y_2,z_2,t_2\},\\
&C_{b_3}=\{x_2,y_1,z_2,t_3\},\\
&C_{b_4}=\{x_2,y_3,z_3,t_1\},\\
&C_{b_5}=\{x_3,y_2,z_3,t_3\},\\
&C_{b_6}=\{x_3,y_3,z_1,t_2\}\}.
\end{array}
\end{eqnarray*}
\begin{lemma}\label{lem:sd8}
Suppose $3\nmid a$, then the constructed  $(a^2,4,a)$-regular  and $(a^2-a,4,\{a-1,a\})$-near regular binary LDPC codes have stopping distance at least $8$.
\end{lemma}
\begin{proof}
We need to demonstrate that $S_1$ and $S_2$ given above do not occur in the corresponding LDPC code.

First assume that there exists columns $b_1,\ b_2,\ b_3,\ b_4,\ b_5,\ b_6$ of $H$ such that $C_{b_1}$, $C_{b_2}$, $C_{b_3}$, $C_{b_4}$, $C_{b_5}$, $C_{b_6}$ take the form given in $S_1$.
Then by Lemma \ref{lem:w_c=4} $D(x_1,1)+y_i\equiv z_i\mod a$ for $1\leq i\leq 3$, $D(x_2,1)+y_1\equiv z_2\mod a$, $D(x_2,1)+y_2\equiv z_3\mod a$ and $D(x_2,1)+y_3\equiv z_1\mod a$. Summing these equivalences implies  $3D(x_1,1)\equiv 3D(x_2,1)\mod a$. If $3\nmid a$ then $D(x_1,1)=D(x_2,1)$, leading to a contradiction.

Next assume that there exists columns $b_1,\ b_2,\ b_3,\ b_4,\ b_5,\ b_6$ of $H$ such that $C_{b_1},C_{b_2},C_{b_3},C_{b_4},C_{b_5},C_{b_6}$ take the form given in $S_2$. Then by Lemma \ref{lem:w_c=4}
 $D(x_1,1)+y_1\equiv D(x_3,1)+y_3 \mod a$, $D(x_2,1)+y_1\equiv D(x_1,1)+y_2\mod a$, $D(x_3,1)+y_2\equiv D(x_2,1)+y_3 \mod a$. Summing these equivalences we obtain $2y_1\equiv 2y_3 \mod a$.
 Similarly we have $t_1-y_1\equiv t_2-y_2\ \mod a$, $t_3-y_1\equiv t_1-y_3 \mod a$ and $t_2-y_3\equiv t_3-y_2 \mod a$. Summing these equivalences we obtain $2y_1\equiv 2y_2 \mod a$. So we have $2y_1\equiv 2y_2\equiv 2y_3 \mod a$. This is a contradiction since it implies $y_1=y_2$, $y_1=y_3$ or $y_2=y_3$.
 \end{proof}
 \begin{corollary}\label{MinDis}
The constructed $(a^2,4,a)$-regular and the $(a^2-a,4,\{a-1,a\})$-near regular binary LDPC codes have minimum distance at least $8$.
\end{corollary}
\begin{proof}
By  definition every set of linearly dependent columns sums to zero modulo 2, thus this set of columns intersects a row in an even number of ones (possibly zero ones). Thus any set of linearly dependent columns forms a stopping set and so the stopping distance of these LDPC codes is a lower bound for the minimum distance. If $3\nmid a$, then the stopping distance of these LDPC codes is at least $8$ by Lemma \ref{lem:sd8}, implying  the minimum distance of the LDPC code is at least $8$. If $3\mid a$ then by the proof of Lemma \ref{lem:sd8} $S_2$ cannot be a subset of the columns of the parity-check matrix of the LDPC. Furthermore, $S_1$ cannot define a linearly independent set as  $S_1$ intersects rows $x_1$ and $x_2$ an odd number of times. Hence the minimum distance of the code is also at least $8$ when $3\mid a$.
\end{proof}
A computer search shows that, for $a\leq 26$, all classified non-isomorphic DCA$(3;a)$ constructed in \cite{DDHKR} produce $(a^2-a,4,\{a-1,a\})$-near regular binary LDPC codes that have minimum distance 8.

It is also possible to prove that the infinite family of LDPC codes constructed from the DCA$(3;a)$ given by Equation \eqref{eq:DCA}
have minimum distance $8$.

Let $a$ be even. It is known that the matrix
$D=[D(j,g)]$, where
\begin{eqnarray}\label{P-DCA}
D(j,g)=\left\{\begin{array}{ll}
0,&\mbox{if }g=0\\
j,&\mbox{if }g=1 \\
\left\{\begin{array}{ll}
2j+1   & \mbox{for } 0 \leq j \leq \frac{a}{2}-1,\\
2(j-\frac{a}{2})  &  \mbox{for } \frac{a}{2} \leq j \leq a-1,
\end{array}\right. & \mbox{if }g=2
\end{array}
\right.\label{eq:DCA}
\end{eqnarray}
forms a DCA$(3;a)$.

\begin{lemma} Suppose $D=[D(i,j)]$ is the DCA$(3;a)$ given by Equation \eqref{eq:DCA}. Then the
$(a^2-a,4,\{a-1,a\})$-near regular binary LDPC with parity-check matrix ${\overline H}$ has minimum distance $8$.
\end{lemma}
\begin{proof}
We know by Corollary \ref{MinDis} that the minimum distance is at least $8$.
It is also easy to see that columns of $\overline{H}$ corresponding to the blocks given below sum to zero mod 2 and thus form a linearly dependent set of columns.
\begin{eqnarray*}
C_{b_1}&=&\{0,a+1,2a+1,3a+2\},\\
C_{b_2}&=&\{0,3a/2-2,5a/2-2,7a/2-1\},\\
C_{b_3}&=&\{1,3a/2-2,5a/2-1,7a/2+1\}, \\
C_{b_4}&=&\{1, 2a-1,2a,3a+2\},\\
C_{b_5}&=&\{a/2-2,a+1,5a/2-1,4a-2\}, \\
C_{b_6}&=&\{a/2-2,3a/2+2,2a,7a/2-1\},\\
C_{b_7}&=&\{a/2-1,3a/2+2,2a+1,7a/2+1\}, \\
C_{b_8}&=&\{a/2-1,2a-1,5a/2-2,4a-2\},
\end{eqnarray*} where $b_1=1$, $b_2=\frac{a}{2}-2$, $b_3=\frac{3a}{2}-2$, $b_4=2a-1$, $b_5=a(\frac{a}{2}-2)+1$, $b_6=a(\frac{a}{2}-2)+\frac{a}{2}+2$, $b_7=a(\frac{a}{2}-1)+\frac{a}{2}+2$, $b_8=a(\frac{a}{2}-1)+a-1$. Hence, there exists a code word of weight $8$. \end{proof}

Thus we have proved the following theorem:
\begin{theorem}\label{LDPC_pseudo}
Let ${\overline H}$ be the parity-check matrix based on the DCA$(3;a)$ given by Equation \eqref{eq:DCA}. Then ${\overline H}$ is the parity-check matrix of a
$(a^2-a,4,\{a-1,a\})$-near regular binary LDPC code of length $a^2 - a$, girth at least 6,  rate at least $1-(4a-6)/(a^2 - a)$ and   minimum distance $8$.
\end{theorem}

Let $a$ be odd. It can be shown that for $k=3$ and $\alpha$ satisfying gcd$(\alpha,a)=1$ and gcd$(\alpha-1,a)=1$, the array
$D=[D(j,g)]$, where
\begin{eqnarray}\label{eqDM}
D(j,g)=\left\{\begin{array}{ll}
0,&\mbox{if }g=0\\
j,&\mbox{if }g=1 \\
\alpha j&\mbox{if }g=2\\
\end{array}
\right.
\end{eqnarray}
forms a DM$(3;a)$ in standard form. Choosing $\alpha=2$ will give a DM$(3;a)$ for all odd $a$. But for better performance, we will assume that $3\nmid a$, $5\nmid a$ and choose $\alpha=(a-1)/2$. When $3\nmid a$, $5\nmid a$ and $\alpha=(a-1)/2$ the constructed codes have improved stopping distance and minimum distance.
\begin{lemma}
 Assume $a$ is an odd positive integer satisfying gcd$(a,3)=1$ and gcd$(a,5)=1$. Let $H$ be the parity-check matrix based on the DM$(3;a)$ given by Equation \eqref{eqDM} with $\displaystyle \alpha=\frac{a-1}{2}$. Then $H$ is the parity-check matrix of a $(a^2,4,a)$-regular binary QC-LDPC code with  minimum distance $10$.
\end{lemma}
\begin{proof}
In Theorem \ref{th:QC} (given in the appendix) we show that $H$ can be put into the quasi-cyclic form.
Furthermore in Theorem \ref{Min10} we show that  the minimum distance is at least $10$. We now give a dependent set of $10$ columns of $H$ to show that the minimum distance is exactly $10$.

Let $\alpha=(a-1)/2$ and $\alpha^{-1}$ be the multiplicative inverse of $\alpha$ in $mod\ a$ so that $\alpha.\alpha^{-1}\equiv 1\ mod\ a$. Observe that $\alpha^{-1}\not\equiv 1\ mod\ a$ and $\alpha^{-1}\not\equiv 2\ mod\ a$.  Now it can be seen that the columns of $H$ corresponding to the blocks given below sum to zero mod 2 and thus form a linearly dependent set of columns.
\begin{eqnarray*}
C_{b_1}&=&\{0,a+2,2a+2,3a+2\}, \\
C_{b_2}&=& \{0,a+\alpha+1,2a+\alpha+1,3a+\alpha+1\},\\
C_{b_3}&=&\{1,a+2,2a+3,3a+\alpha+2\}, \\
C_{b_4}&=&\{1,\alpha,\alpha+1,4a-1\},\\
C_{b_5}&=&\{2,a,2a+2,4a-1\}, \\
C_{b_6}&=&\{2,a+1,2a+3,3a\},\\
C_{b_7}&=&\{\alpha^{-1},a+1,2a+\alpha^{-1}+1,3a+2\}, \\
C_{b_8}&=&\{\alpha^{-1},a+\alpha+1,2a+\alpha^{-1}+\alpha+1,3a+\alpha+1\},\\
C_{b_9}&=&\{\alpha^{-1}+1,a,2a+\alpha^{-1}+1,3a+\alpha+1\}, \\
C_{b_{10}}&=&\{\alpha^{-1}+1,a+\alpha,2a+\alpha^{-1}+\alpha+1,3a\},
\end{eqnarray*}
where $b_1=2$, $b_2=\alpha+1$, $b_3={a+2}$, $b_4=a+\alpha$, $b_5=2a$, $b_6=2a+1$, $b_7=\alpha^{-1}a+1$, $b_8=\alpha^{-1}a+\alpha+1$, $b_9=(\alpha^{-1}+1)a$ and $b_{10}=(\alpha^{-1}+1)a+\alpha$. Hence, there exists a code word of weight $10$.
\end{proof} Thus we have proved the following theorem:
\begin{theorem}\label{LDPC_pseudo2}
 Assume $a$ is an odd positive integer satisfying $(a,3)=1$ and $(a,5)=1$. Let $H$ be the parity-check matrix based on the DM$(3;a)$ given by Equation \eqref{eqDM} where $\alpha=(a-1)/2$. Then $H$ is the parity-check matrix of a
$(a^2,4,a)$-regular binary QC-LDPC code of length $a^2$, girth at least $6$,  rate equal to $\displaystyle 1-\frac{4a-3}{a^2}$ and  minimum distance $10$.
\end{theorem}
The matrix $H^*$ is the general quasi-cylic form of the parity check matrix for the code obtained in the above theorem.
\begin{eqnarray*}
H^*=\left[
\begin{array}{ccccccccccccc}
I&I&I&I&...&I\\
I&P^1&P^2&P^3&...&P^{a-1}\\
I&P^{2^{-1}}&P^{2*2^{-1}}&P^{3*2^{-1}}&...&P^{(a-1)*2^{-1}}\\
I&P^{(\alpha+1)^{-1}}&P^{2*(\alpha+1)^{-1}}&P^{3*(\alpha+1)^{-1}}&...&P^{(a-1)*(\alpha+1)^{-1}}\\
\end{array}\right].
\end{eqnarray*}

\section{ Performance Analysis}\label{PA}

In this section we present the error correcting performance of the proposed LDPC codes via simulations. In addition to the code rate and minimum distance, simulations can provide another indicator for the performance of an LDPC code. Here simulations have been conducted  over two different channels, AWGN and BEC. For clarity, we label  the codes for comparison using the quadruple $[m, R, w_c, w_r]$ consisting of code length $m$, code rate $R$, column weight $w_c$ and row weight $w_r$. For irregular or near-regular LDPC codes, the  row weight is given as $a-1$ or taken as an average which is indicated by $\sim$. Also to reduce verbiage we will refer to an $(a^2,4,a)$-regular binary QC-LDPC code constructed in Construction 1 using a DM$(3;a)$ as a  DM$(3;a)$ code and an $(a^2-a,4,\{a-1,a\})$-near regular binary LDPC code constructed in Construction 2 using a DCA$(3;a)$ as a  DCA$(3;a)$ code. In both channels, we compare the codes given in Table \ref{comparison-codes}. The  parity-check matrices are illustrated in Figures \ref{fig:dm43} - \ref{fig:mackay&neal} in Appendix \ref{PCM}.

\begin{table}[ht]
\caption{LDPC codes that were compared for BER and FER performance.}\label{comparison-codes}
\begin{center}
\begin{tabular}{|l|l|l|}
\hline
LDPC Code Type&$[m,R,w_c,w_r]$&Reference\\
\hline\hline
DM$(3;43)$&$[1849,0.91,4,43]$&Current paper\\
DCA$(3;44)$&$[1892,0.91,4,43]$&Current paper\\
TD-LDPC &$[1849,0.91,4,43]$&\cite{Gruner2013}\\
Lattice &$[1849,0.91,4,43]$&\cite{VasMil04}\\
Gallager  &$[1849,0.91,4,43]$&\cite{Gallager63}\\
PEG &$[1849,0.91,4,\sim 43]$ &\cite{XiaoPEG} \\
Mackay\&Neal &$[1908,0.89,4,36]$&\cite{MacKayNeal97} \\
\hline
\end{tabular}
\end{center}
\end{table}

These codes were chosen as they all have similar parameters, namely column weight $4$ and their rates are in the range $0.89$ to $0.91$. In addition, this presents the opportunity to compare codes with a very narrow range of lengths. Like the proposed codes, the transversal design code (termed TD-LDPC) as constructed in  \cite{Gruner2013} and  the code constructed using lattices (termed Lattice) as in  \cite{VasMil04} are structured LDPC codes, namely, they are constructed using certain combinatorial designs, finite geometries, etc. The other three codes, however, are pseudo-random in nature. The PEG code given here, is constructed via open-source software given in \cite{Lcrypto}. As for Mackay\&Neal code, we use the \emph{Encyclopedia of Sparse Graph Codes}, a database of sparse graph codes, written by David J.C. MacKay. Moreover, the simulations were performed via an open-source library called AFF3CT \cite{Cassagne2019a}, a toolbox dedicated to forward error correction, written in C++.

\subsection{The Additive White Gaussian Noise Channel (AWGN) }

The analysis has been performed through the transmission of the zero vector using binary-phase shift-key (BPSK) modulation over varying signal-to-noise ratios (SNRs, Eb/No) assuming transmission over AWGN channel. Since only binary messages are being transmitted, we chose the zero vector which allows errors to be added randomly across the entire vector. We used the belief propagation (BP) based decoding algorithm given in \cite{YPNA01} with the \emph{sum-product algorithm} (SPA) implementation. The procedure is iterated until the zero vector is obtained or a maximum number of iterations (100) is reached. Also, at each SNR level, we monitor the analysis until it reaches 50 wrongly decoded vectors. Figures \ref{fig:berawgn} and  \ref{fig:ferawgn} illustrate the decoding BER and FER performance of the codes in Table \ref{comparison-codes}.

\begin{figure}
	\centering
	\includegraphics[scale=0.7]{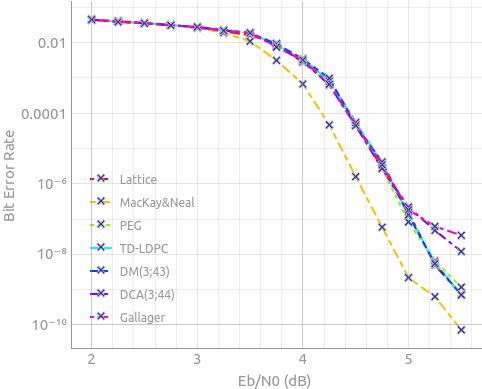}
	\caption{ BER comparison of the codes in Table \ref{comparison-codes} over AWGN.}
	\label{fig:berawgn}
\end{figure}

\begin{figure}
	\centering
	\includegraphics[scale=0.7]{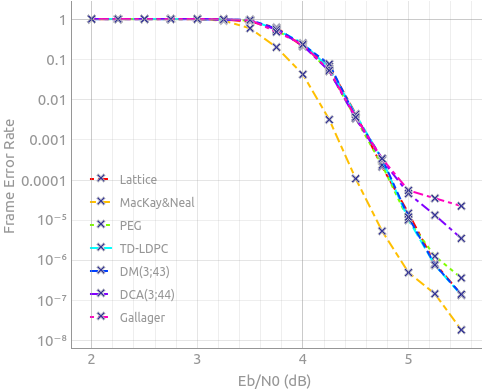}
	\caption{ FER comparison of the codes in Table \ref{comparison-codes} over AWGN.}
	\label{fig:ferawgn}
\end{figure}

As can be observed from Figures \ref{fig:berawgn} and  \ref{fig:ferawgn}, the performance of the DM$(3;43)$ code is better than the performance of the DCA$(3;44)$ code at higher Eb/No values. Theorem \ref{LDPC_pseudo} and \ref{LDPC_pseudo2} ensure that the minimum distances of the DM$(3;43)$ code and the DCA$(3;44)$ code are 10 and 8, respectively. Therefore, such a difference in their performances can be expected.

The other two structured codes, namely Lattice and TD-LDPC, have exactly the same parameters as the DM$(3;43)$ code, with Figures \ref{fig:dm43}, \ref{fig:td-ldpc}, and \ref{fig:lattice} showing that they are similar in the sparsity pattern of their matrices. Thus, they perform similarly as expected. However, the proposed code still has the advantage of its algebraic properties. Unlike the other two codes, we can calculate the code rate and the minimum distance precisely.

Considering FER after the 4.75 Eb/No level, the proposed codes outperform the Gallager code. In addition, the DM$(3;43)$ code and the PEG code perform similarly, with our code performing a bit better than the PEG code at 5.5 Eb/No level. Both BER and FER graphs show that the Mackay\&Neal code outperforms the other codes. The high-rate codes in the Encyclopedia of Sparse Graph Codes are generally known as having high minimum distance, but it's hard to determine the precise values (finding the minimum distance of an LDPC code in general is an NP-hard problem \cite{XiaoNP}). However, achieving a similar or better performance compared to some randomly generated codes, is promising, for instance as a further step, the proposed codes can be improved to have a higher minimum distance.

\subsection{Binary Erasure Channel (BEC) }
Additionally, an analysis has been conducted through the transmission of the zero vector using on-off keying (OOK) modulation over the BEC under varying error probabilities. For the decoder we used the belief propagation (BP) based decoding algorithm given in \cite{YPNA01} with the \emph{normalized minimum-sum} (NMS) implementation \cite{CF02}. As in the previous case, the maximum number of iterations is 100, and the maximum number of block errors 50. Figures \ref{fig:berbec} and  \ref{fig:ferbec} illustrate the decoding BER and FER performance of the codes in Table \ref{comparison-codes}.
\begin{figure}
	\centering
	\includegraphics[scale=0.7]{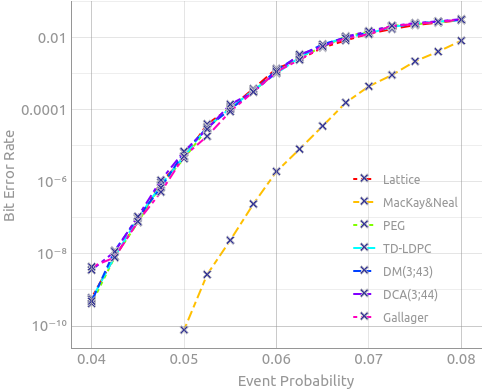}
	\caption{ BER comparison of the codes in Table \ref{comparison-codes} over BEC.}
	\label{fig:berbec}
\end{figure}
\begin{figure}
	\centering
	\includegraphics[scale=0.7]{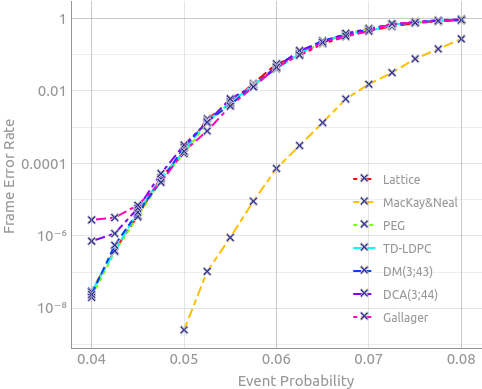}
	\caption{ FER comparison of the codes in Table \ref{comparison-codes} over BEC.}
	\label{fig:ferbec}
\end{figure}
The results are almost the same as those obtained over AWGN channel. While the DM$(3;43)$ code, the Lattice code, the TD-LDPC code, and the PEG code perform similarly, they outperform the DCA$(3;44)$ code and the Gallager code. Also, the Mackay\&Neal code performs much better than the others. It can be concluded that the proposed codes have similar stopping set size/distribution with some others. As in \cite{Gruner2013}, additional advantage is possible in that the proposed codes have the potential for  larger stopping sets. Such improvements can be investigated by discarding some rows in a DM$(3;a)$ or DCA$(3;a)$, as done in Construction 2.
\section{Conclusion}\label{Conclusion}
An explosion in the number of smart devices requires new generations of error correcting codes to meet the demand for ultra-reliable and low latency smart object communication. Furthermore, 5G networks need codes supporting variable code rates and lengths. Both LDPC and polar codes promise the requisite functionality and are currently being widely researched.

In this paper we  presented two new constructions of LDPC codes developed from difference matrices and difference covering arrays. When compared to previous constructions, the constructions presented here leverage the advantage of the underlying algebraic structure, which is the cyclic group with binary operation addition modulo an integer. These   algebraic structures, difference matrices and difference covering arrays exist for all orders $a$, allowing construction of an infinite family of LDPC codes and theoretical verification of the properties of the codes. In particular, for $a$ even, we presented LDPC codes with lengths at least $a^2-a$ and rate  at least $1-\frac{4a-6}{a^2-a}$. Similarly, for $a$ odd, we constructed  LDPC codes with length $a^2$ and  rate $1-\frac{4a-3}{a^2}$. Furthermore, for $a$ odd, we showed that the constructed codes are quasi-cyclic and provided $a$ is not divisible by $3$ or $5$, the codes have minimum distance at least $10$.
The simulation results presented in this paper, using standard decoding algorithms, showed that these LDPC codes perform well enough when compared to previous constructions of LDPC codes similar to ours, as well as to some randomly generated codes.

\backmatter

\bmhead{Acknowledgments}

E. Sule Yazici would like to thank RMIT University for travel support. This research was carried out during that visit.

\begin{appendices}

\section{Minimum Distance}\label{MinDist}
\begin{theorem}\label{Min10}
Let $a>3$ and  $H$ be the parity-check matrix based on the DM$(3;a)$ given by Equation \eqref{eqDM} with $\alpha=\frac{a-1}{2}$ where gcd$(a,3)=1$ and gcd$(a,5)=1$. Then $H$ is the parity-check matrix of a $(a^2,4,a)$-regular binary LDPC code with  minimum distance at least $10$.
\end{theorem}

\begin{proof}
To reduce excessive notation in this proof, all equalities will be assumed to be equivalences modulo $a$.

We will show that minimum distance of the constructed code is at least $10$. First observe that the minimum distance cannot be odd. Each column contains exactly one $1$ in the first $a$ rows and each row should have an even number of $1$'s in these columns so the number of columns in the linearly dependent set of columns should be even. If we can show that there are no $8$ columns that are linearly dependent then this would imply the minimum distance is at least $10$

Assume  $a$ is odd and the parity-check matrix $H$ contains $8$ columns that are linearly dependent, where the bitwise sum over the columns is taken modulo $2$. Then by Lemma \ref{lem:w_c=4} these columns take the form
\begin{eqnarray*}
\begin{array}{cc}
\{x_1,y_1,z_1,t_1\},&\{x_5,y_5,z_5,t_5\},\\
\{x_2,y_2,z_2,t_2\},&\{x_6,y_6,z_6,t_6\},\\
\{x_3,y_3,z_3,t_3\},&\{x_7,y_7,z_7,t_7\},\\
\{x_4,y_4,z_4,t_4\},&\{x_8,y_8,z_8,t_8\},
\end{array}
 \end{eqnarray*}
where  $x_i+y_i=z_i\ mod\ a$ and $\ \alpha x_i+y_i=t_i\ mod\ a$ for $i=1,\dots, 8$ and $\alpha=\frac{a-1}{2}$. Note that, as $a$ is odd, gcd$(a,\alpha+1)=1$; since gcd$(a,3)=1$ we have gcd$(\alpha-1,a)=1$ and as gcd$(a,5)=1$ we have gcd$(\alpha-2,a)=1.$
 Furthermore, note that since $a$ is odd, $2k=2l\ mod\ a$  implies $k=l$ modulo $a$ in general.

Under the assumption that these columns are linearly dependent, it follows that all elements $x_i,y_i,z_i,t_i$, for $1\leq i\leq 8$, occur an even number of times, with the RC-constraint implying these elements each occur either $2$ or $4$ times.

 Assume, without loss of generality (wlog), that $x_1$ occurs $4$ times. Then the RC-constraint implies there are $y_1$, $y_2$, $y_3$ and $y_4$ all distinct, each occurring exactly twice, similarly for $z_1,\ z_2,\ z_3,\ z_4$ and $t_1,\ t_2,\ t_3,\ t_4$. Note that since $a$ is odd, the equations $x_1+y_1=z_1$, $x_1+y_2=z_2$, $x_2+y_2=z_1$ and $x_2+y_1=z_2$ together results in a contradiction.

Thus there are two possibilities:

(i) Either it may be assumed that there exists $x_2$ and $x_3$ not necessarily distinct such that $x_2+y_2=z_1$, $x_2+y_3=z_2$, $x_3+y_4=z_3$ and $x_3+y_1=z_4$. Consequently $z_2-z_1=z_3-z_2$ and $z_4-z_3=z_1-z_4$, which gives $2z_2=z_1+z_3=2z_4$, contradicting the fact that $z_2$ and $z_4$ are distinct modulo $a$.

(ii) Or it may be assumed that there exists distinct $x_2$ and $x_3$ such that $x_2+y_1=z_3$, $x_2+y_2=z_4$, $x_3+y_3=z_1$ and $x_3+y_4=z_2$. This case is a special case of Case 1-a-) below where we set $x_1=x_4$ and it results in a contradiction.

A similar argument will show that it is not possible for any $y_i$ to occur $4$ times.

 Next assume that, for all $1\leq i\leq 4$, each $x_i$ and $y_i$ occurs exactly twice,  leading to two non-isomorphic subcases.

\begin{eqnarray}
\begin{array}{ccc}
Case\ 1-)&\mbox{or}&Case\ 2-)\\
\begin{array}{c}
\{x_1,y_1,z_1,t_1\}\\
\{x_1,y_2,z_2,t_2\}\\
\{x_2,y_1,z_3,t_3\}\\
\{x_2,y_2,z_4,t_4\}\\
\{x_3,y_3,z_5,t_5\}\\
\{x_3,y_4,z_6,t_6\}\\
\{x_4,y_3,z_7,t_7\}\\
\{x_4,y_4,z_8,t_8\}
\end{array}&&
\begin{array}{c}
\{x_1,y_1,z_1,t_1\}\\
\{x_1,y_2,z_2,t_2\}\\
\{x_2,y_2,z_3,t_3\}\\
\{x_2,y_3,z_4,t_4\}\\
\{x_3,y_3,z_5,t_5\}\\
\{x_3,y_4,z_6,t_6\}\\
\{x_4,y_4,z_7,t_7\}\\
\{x_4,y_1,z_8,t_8\}
\end{array}
\end{array}\label{eq:cases2b}
\end{eqnarray}

\noindent {\bf CASE 1-)}

First observe that, wlog, $x_1+y_1=z_1$, $x_1+y_2=z_2$, $x_2+y_1=z_2$ and $x_2+y_2=z_1$ will imply $z_1-z_2=y_1-y_2=z_2-z_1$, a contradiction. Hence we may assume that $\mid\hspace{-0.1cm}\{z_1,z_2,z_3,z_4\}\hspace{-0.1cm}\mid,\mid\hspace{-0.1cm}\{z_5,z_6,z_7,z_8\}\hspace{-0.1cm}\mid\geq 3$, and similarly $\mid\hspace{-0.1cm}\{t_1,t_2,t_3,t_4\}\hspace{-0.1cm}\mid,\mid\hspace{-0.1cm}\{t_5,t_6,t_7,t_8\}\hspace{-0.1cm}\mid\geq 3$.

Furthermore the description given above implies
\begin{eqnarray}
(x_2-x_1)&=&z_3-z_1=z_4-z_2,\label{eq:gen3}\\
(x_4-x_3)&=&z_7-z_5=z_8-z_6,\label{eq:gen4}\\
\alpha(x_2-x_1)&=&t_3-t_1=t_4-t_2,\label{eq:gen5}\\
\alpha(x_4-x_3)&=&t_7-t_5=t_8-t_6.\label{eq:gen6}\\
z_1+z_4=z_2+z_3,&&z_5+z_8=z_6+z_7,\label{eq:genz12}\\
t_1+t_4=t_2+t_3,&&t_5+t_8=t_6+t_7.\label{eq:gent12}
\end{eqnarray}
Now assume, wlog, $\mid\{z_1,z_2,z_3,z_4\}\mid=3$ and $z_1=z_4$.
 Equation \eqref{eq:genz12} gives $2z_1=z_2+z_3$. We may deduce  $\mid\{z_5,z_6,z_7,z_8\}\mid=3$ and, wlog, $z_5=z_8$ and $z_6=z_2$ so $z_7=z_3$. Equation \eqref{eq:genz12} implies $2z_5=z_2+z_3$. Now combining this information gives
 $2(z_1-z_5)=0$ and so  $z_1=z_4=z_5=z_8$, leading to $t_1,\ t_4,\ t_5,\ t_8$ being distinct. Now, assuming that $t_5=t_2$ and $t_8=t_3$ then either $(t_6,t_7)=(t_1,t_4)$ or $(t_6,t_7)=(t_4,t_1)$.

The former  implies $z_1-z_3=(x_1-x_2)=y_2-y_1=t_2-t_1=y_3-y_4=(x_4-x_3)=z_3-z_1$ a contradiction.

The latter implies $(x_2-x_1)=z_3-z_1=(x_4-x_3)$, so $t_4-t_2=t_1-t_2$ which leads to a contradiction since $t_1$ and $t_4$ are distinct.

The case $t_5=t_3$ and $t_8=t_2$ follows similarly.

Also similarly $\mid\{t_1,t_2,t_3,t_4\}\mid=3$  is not possible.

So  $\mid\{z_1,z_2,z_3,z_4\}\mid=\mid\{z_5,z_6,z_7,z_8\}\mid=\mid\{t_1,t_2,t_3,
t_4\}\mid=\mid\{t_5,t_6,t_7,t_8\}\mid=4.$

Hence in Case 1-) we may assume WLOG $z_5=z_1$.

\noindent {\bf Case  1-)a-)} Assume $z_6=z_2$.

If $z_7=z_4$, then $z_8=z_3$ and so Equation \eqref{eq:genz12} gives $z_1+z_4=z_2+z_3$ and $z_1+z_3=z_2+z_4$, implying $z_2+z_3-z_4=z_2+z_4-z_3$ and leading to the contradiction $z_3=z_4$ since $a$ is odd.  Thus $z_7=z_3$ and $z_8=z_4$.

Then we have
\begin{eqnarray}
 t_2-t_1=y_2-y_1=z_2-z_1=y_4-y_3=z_4-z_3\nonumber\\
 =t_4-t_3=t_6-t_5=t_8-t_7.\label{eq:Case1a}
\end{eqnarray}

Consider the case where $(t_5,t_8)=(t_i,t_1)$, for $i\in\{2,3\}$ then $(t_6,t_7)=(t_j,t_4)$ or $(t_6,t_7)=(t_4,t_j)$ where $j\in\{2,3\}\setminus\{i\}$. Equation \eqref{eq:gent12} implies $t_i+t_1=t_j+t_4$ and $t_1+t_4=t_2+t_3=t_i+t_j$. Combining these equations  gives $t_4-t_i=t_i-t_4$,  a contradiction since $t_4$ and $t_i$ are distinct.

Thus, we have the following possibilities.

   {\bf i-)} $(t_5,t_8)=(t_4,t_1)$ and $(t_6,t_7)=(t_3,t_2)$.
  By Equation \eqref{eq:Case1a} $t_2-t_1=t_1-t_2$, a contradiction.

   {\bf ii-)} $(t_5,t_8)=(t_2,t_3)$ and $(t_6,t_7)=(t_1,t_4)$. By Equation \eqref{eq:Case1a},
  $t_1-t_2=t_2-t_1$, a contradiction.

   {\bf iii-)} $(t_5,t_8)=(t_2,t_3)$ and $(t_6,t_7)=(t_4,t_1)$. By Equation \eqref{eq:Case1a},
  $t_2-t_1=t_4-t_2$ implying $2t_2=t_1+t_4$ combining with $t_1+t_4=t_2+t_3$ we have $t_2=t_3$, a contradiction.

 {\bf iv-)} $(t_5,t_8)=(t_3,t_2)$ and $(t_6,t_7)=(t_1,t_4)$. Then by Equation \eqref{eq:Case1a},
  $t_2-t_1=t_1-t_3$ implying $2t_1=t_2+t_3$ combining with $t_1+t_4=t_2+t_3$ we have $t_1=t_4$ a contradiction.

   {\bf v-)} $(t_5,t_8)=(t_3,t_2)$ and $(t_6,t_7)=(t_4,t_1)$.  Then by Equations \eqref{eq:gen3} and \eqref{eq:gen4}, we have $(x_2-x_1)=z_3-z_1=z_7-z_5=(x_4-x_3)$. Hence by Equations \eqref{eq:gen5} and \eqref{eq:gen6},  $t_3-t_1=\alpha(x_2-x_1)=\alpha(x_4-x_3)=t_7-t_5=t_1-t_3$, a contradiction.

\noindent {\bf Case 1-)b-)} Assume $z_6=z_3$.

If $z_7=z_4$, $z_2=z_8$ then Equation \eqref{eq:genz12} gives
$z_1+z_4=z_2+z_3$ and $z_1+z_2=z_3+z_4$, implying $2z_2=2z_4$, a contradiction. Thus $z_7=z_2$ and $z_8=z_4$.

Then we have \begin{eqnarray}
 t_4-t_3=t_2-t_1=y_2-y_1=z_2-z_1=z_4-z_3\label{eq:Case1b}\\
 y_4-y_3=z_3-z_1=z_4-z_2=t_6-t_5=t_8-t_7.\label{eq:Case1bi}
\end{eqnarray}
Consider the case where $(t_5,t_8)=(t_i,t_1)$, for $i\in\{2,3\},$ then $(t_6,t_7)=(t_j,t_4)$ or $(t_6,t_7)=(t_4,t_j)$ where $j\in\{2,3\}\setminus\{i\}$.  Equation \eqref{eq:gent12} implies $t_i+t_1=t_j+t_4$ and $t_1+t_4=t_2+t_3=t_i+t_j$. Combining these equations gives $t_4-t_i=t_i-t_4$, a contradiction.

So we have the following possibilities:

  {\bf i-)} $(t_5,t_8)=(t_4,t_1)$ and $(t_6,t_7)=(t_2,t_3)$.

 By Equations \eqref{eq:gen3}, \eqref{eq:gen5} and \eqref{eq:Case1bi}, $x_2-x_1=z_4-z_2$ and $\alpha(x_2-x_1)=t_3-t_1=z_2-z_4$, giving $(x_2-x_1)=-\alpha(x_2-x_1)$ implying $0=(\alpha+1)(x_2-x_1)$. Now as $\alpha+1=\frac{a+1}{2}$ and gcd$(\frac{a+1}{2},a)=1$, we have $x_2=x_1$ a contradiction.

   {\bf ii-)} $(t_5,t_8)=(t_2,t_3)$ and $(t_6,t_7)=(t_1,t_4)$.
  By Equation \eqref{eq:Case1bi}, $z_4-z_3=t_4-t_3=z_2-z_4$, implying $2z_4=z_2+z_3$ and $z_1-z_2=t_1-t_2=z_3-z_1$, implying $2z_1=z_2+z_3$. Hence $2z_1=2z_4$, a contradiction.

   {\bf iii-)}$(t_5,t_8)=(t_2,t_3)$ and $(t_6,t_7)=(t_4,t_1)$.

   By Equations \eqref{eq:gen3}, \eqref{eq:gen5} and \eqref{eq:Case1bi}, $(x_2-x_1)=z_4-z_2$ and $\alpha(x_2-x_1)=t_3-t_1=z_4-z_2$ giving $(x_2-x_1)=\alpha(x_2-x_1)$ so $(\alpha-1)(x_2-x_1)=0$. Now as $\alpha-1=\frac{a-3}{2}$ and $(a,3)=1$, we have $(\alpha-1,a)=1$ and $x_2=x_1$ a contradiction.

     {\bf iv-)} $(t_5,t_8)=(t_3,t_2)$ and $(t_6,t_7)=(t_1,t_4)$.

    By Equations \eqref{eq:gen3}, \eqref{eq:gen5} and \eqref{eq:Case1bi} $(x_2-x_1)=z_4-z_2$ and $\alpha(x_2-x_1)=t_3-t_1=z_2-z_4$ then $(x_2-x_1)=-\alpha(x_2-x_1)$ and $0=(\alpha+1)(x_2-x_1)$. Now as $\alpha+1=\frac{a+1}{2}$ and $(\frac{a+1}{2},a)=1$, we have $x_2=x_1$ a contradiction.

      {\bf v-)} $(t_5,t_8)=(t_3,t_2)$ and $(t_6,t_7)=(t_4,t_1)$.

      By Equation \eqref{eq:Case1bi}, $z_4-z_3=t_4-t_3=z_3-z_1$, implying $2z_3=z_1+z_4$ and  $z_2-z_1=t_2-t_1=z_4-z_2$, implying $2z_2=z_1+z_4$ and leading to a contradiction.

\noindent {\bf Case 1-)c-)} Assume $z_6=z_4$.

If $z_7=z_2$, $z_8=z_3$ then Equation \eqref{eq:genz12} gives
$z_1+z_3=z_2+z_4$ and $z_1+z_4=z_3+z_2$ hence $2z_4=2z_3$, a contradiction.

If $z_7=z_3$, $z_8=z_2$ then Equation \eqref{eq:genz12} gives
$z_1+z_4=z_2+z_3$ and $z_1+z_2=z_3+z_4$ hence $2z_2=2z_4$, a contradiction.

  \noindent {\bf CASE 2-)}

Considering the list of entries $z_1,\dots,z_8$ in Case 2-), as given in Equation \eqref{eq:cases2b}, these entries can be written as a cyclic list, that is, $(z_1,z_2,z_3,z_4,z_5,z_6,z_7,z_8)$. Any even shift will be isomorphic to this list in nature and an odd shifts will interchange $x_i$'s with $y_i$'s in the equations. There exists 4 sets of pairs  $i,j$, $1\leq i<j\leq 8$ such that $z_i=z_j$, where we may say that $z_i$ and $z_j$ are distance $\mid j-i\mid$ apart.
The cyclic nature of the list implies that for all pairs $i,j$ such that $z_i=z_j$ we may take $\mid j-i\mid\leq 4$. Assume $z_1$ has the smallest distance among the $z_i$'s.

 We address the possibilities with the following subcases:

{\bf a-)} Assume $z_j=z_1$ where $j-1=2$ and $z_{j^\prime}=z_2$ where  $j^\prime-2=2$, implying cyclic list
$(z_1,z_2,z_1,z_2,z_3,z_4,z_3,z_4)$. Then   $y_2-y_1=y_3-y_2$ and $y_4-y_3=y_1-y_4$ implying $2y_2=2y_4$ a contradiction.

 Similarly an odd shift will give a list isomorphic to $(z_2,z_1,z_2,z_3,z_4,z_3,z_4,z_1)$. Then we have $x_1-x_2=x_4-x_1$ and $x_2-x_3=x_3-x_4$ implying the contradiction $2x_1=2x_3$

{\bf b-)} Assume $z_j=z_1$ where $j-1=2$ and $z_{j^\prime}=z_2$ where  $j^\prime-2=3$, with cyclic list
$(z_1,z_2,z_1,z_3,z_2,z_4,z_3,z_4)$. Then
$x_1+y_1=x_2+y_2$,
$x_1+y_2=x_3+y_3$,
$x_2+y_3=x_4+y_4$,
$x_3+y_4=x_4+y_1$,
implying
$y_1-y_2=(x_2-x_3)+y_2-y_3$,
$(x_2-x_3)+y_3-y_4=y_4-y_1$.
Then
$2y_4-y_3-y_1=(x_2-x_3)=y_1-2y_2+y_3=$ or equivalently
$2(y_4+y_2)=2(y_3+y_1)$, so $y_4+y_2=y_3+y_1$. On the other hand, we have $x_1+x_2+y_1+y_3=x_2+x_4+y_2+y_4$, Hence $x_1=x_4$, contradiction.

Similarly an odd shift will give a list isomorphic to $(z_2,z_1,z_3,z_2,z_4,z_3,z_4,z_1)$. Then we have $z_3-z_2=y_2-y_3=x_3-x_1+y_4-y_1$ and $z_4-z_1=y_4-y_1=x_3-x_1+y_3-y_2$ or equivalently $2y_4=2y_1$ a contradiction.

{\bf c-)} Assume $z_j=z_1$ where $j-1=2$ and $z_{j^\prime}=z_2$ where  $j^\prime-2=4$, with cyclic list
$(z_1,z_2,z_1,z_3,z_4,z_2,z_4,z_3)$.
Then
$x_1+y_1=x_2+y_2$,
$x_1+y_2=x_3+y_4$,
$x_2+y_3=x_4+y_1$,
$x_3+y_3=x_4+y_4$,
implying $y_1-y_2=(x_2-x_3)+y_2-y_4$ and $(x_2-x_3)=y_1-y_4$. Then $y_1-y_2=y_1-y_4+y_2-y_4$, so $2y_2=2y_4$ leading to a contradiction.

Similarly an odd shift will give a list isomorphic to $(z_2,z_1,z_3,z_4,z_2,z_4,z_3,z_1)$. Then we have $z_2-z_1=y_1-y_2=x_3-x_4+y_3-y_1$ and $z_4-z_3=y_3-y_2=x_3-x_4$ or equivalently $2y_3=2y_1$ a contradiction.

{\bf d-)} Assume $z_j=z_1$ where $j-1=2$ and $z_{j^\prime}=z_2$ where  $j^\prime-2=4$, with cyclic list
$(z_1,z_2,z_1,z_3,z_4,z_2,z_3,z_4)$. Then
$x_1+y_1=x_2+y_2$,
$x_1+y_2=x_3+y_4$,
$x_2+y_3=x_4+y_4$,
$x_3+y_3=x_4+y_1$.
So $y_1-y_2=(x_2-x_3)+y_2-y_4$ and $(x_2-x_3)=y_4-y_1$.
Then $y_1-y_2=y_4-y_1+y_2-y_4$ or equivalently $y_1-y_2=y_2-y_1$,  a contradiction.

Similarly an odd shift will give a list isomorphic to $(z_2,z_1,z_3,z_4,z_2,z_3,z_4,z_1)$. Then we have $z_2-z_1=y_1-y_2=x_3-x_4+y_3-y_1$ and $z_4-z_3=y_3-y_2=x_4-x_3$ or equivalently $2y_2=2y_1$ a contradiction.

{\bf e-)} Assume $z_j=z_1$ where $j-1=3$ and $z_{j^\prime}=z_2$ where  $j^\prime-2=4$, with cyclic list
$(z_1,z_2,z_3,z_1,z_4,z_2,z_3,z_4)$. Then
$x_1+y_1=x_2+y_3$,
$x_1+y_2=x_3+y_4$,
$x_2+y_2=x_4+y_4$,
$x_3+y_3=x_4+y_1$,
implying $y_1-y_2=(x_2-x_3)+y_3-y_4$, $y_4-y_1=(x_2-x_3)+y_2-y_3$. Then $y_4-y_1=y_1-y_2-y_3+y_4+y_2-y_3$ so $2y_1=2y_3$, leading to a contradiction.

Similarly an odd shift will give a list isomorphic to $(z_2,z_3,z_1,z_4,z_2,z_3,z_4,z_1)$. Then we have $z_2-z_3=y_1-y_2=y_3-y_4$ and $z_4-z_1=y_4-y_1=y_3-y_2$ or equivalently $2y_3=2y_4$ a contradiction.

{\bf f-)}
Assume $z_j=z_1$ where $j-1=4$, giving the cyclic list
$(z_1,z_2,z_3,z_4,z_1,z_2,z_3,z_4)$. Then
$x_1+y_1=x_3+y_3$,
$x_1+y_2=x_3+y_4$,
$x_2+y_2=x_4+y_4$,
$x_2+y_3=x_4+y_1$
implying $y_1-y_2=y_3-y_4$ and $y_2-y_3=y_4-y_1$. Then $y_1-y_3=y_3-y_1$, a contradiction.

Similarly an odd shift will give a list isomorphic to $(z_2,z_3,z_4,z_1,z_2,z_3,z_4,z_1)$. Then we have $z_2-z_3=y_1-y_2=y_3-y_4$ and $z_4-z_1=y_4-y_1=y_2-y_3$ or equivalently $2y_2=2y_4$ a contradiction.

{\bf g-)} $z_j=z_1$ where $j-1=3$ and $z_{j^\prime}=z_2$ where  $j^\prime-2=3$, with cyclic list
$(z_1,z_2,z_3,z_1,z_4,z_3,z_2,z_4)$. Observe that all $z_i$ have distance 3 in the cyclic list. Then
$x_1+y_1=x_2+y_3$,
$x_1+y_2=x_4+y_4$, implying $y_1-y_2=(x_2-x_4)+y_3-y_4$

Finally consider the list of entries $t_1,\dots,t_8$ as given in Equation \eqref{eq:cases2b}. Again, these entries can be written as a cyclic list, that is, $(t_1,t_2,t_3,t_4,t_5,t_6,t_7,t_8)$.

The above arguments can be repeated for this cyclic list where any occurrence of $x_i$ in an equation is replaced by $\alpha x_i$. Thus, it can be argued that the distinct entries $t_1,t_2,t_3,t_4$ have distance $3$ in the above cyclic list. Now considering the RC-constraint we have $t_6=t_1$. Then $t_2=t_5$, $t_3=t_8$ and $t_4=t_7$. Hence we have the list $(t_1,t_2,t_3,t_4,t_2,t_1,t_4,t_3)$. But then we have
$\alpha x_1+y_1=k_2x_3+y_4$,
$\alpha x_1+y_2=k_2x_3+y_3$,
implying $y_1-y_2=y_4-y_3$. Hence  $2(y_1-y_2)=(x_2-x_4)$ and combining with $\alpha x_2+y_2=\alpha x_4+y_1$ we have $2\alpha(x_2-x_4)=(x_2-x_4)$. Hence $(2\alpha-1)(x_2-x_4)=(a-2)(x_2-x_4)=0$, a contradiction since gcd$(a-2,a)=1$.

Similarly an odd shift will give a list isomorphic to $(z_2,z_3,z_1,z_4,z_3,z_2,z_4,z_1)$. And the list of $t_i$ as $(t_2,t_3,t_4,t_2,t_1,t_4,t_3,t_1)$. Then we have $z_2-z_3=y_1-y_2=y_4-y_3$ and  $t_2-t_3=y_1-y_2=\alpha (x_2-x_4)+y_3-y_4$ implying $2(y_1-y_2)=\alpha (x_2-x_4)$ Now combining with $z_1=x_2+y_2=x_4+y_1$ we have $\alpha(y_1-y_2)=2(y_1-y_2)$ or equivalently $(\alpha-2)(y_1-y_2)=0$. But as $\alpha-2=\frac{a-5}{2}$ and gcd$(a,5)=1$ we have $y_1=y_2$ a contradiction.
\end{proof}
The proof of this theorem can be readily generalized for any $\alpha$ with required properties:

\begin{lemma}
Let $H$ be the parity-matrix based on the DM$(3;a)$ given in Equation \eqref{eqDM} , where $(a,\alpha-2+i)=1$ for all $0\leq i\leq 3$; and   $(a,2\alpha-1)=1$. Then $H$ is the parity-check matrix of a $(a^2,4,a)$-regular binary LDPC code with minimum distance at least $10$.
\end{lemma}

Note that gcd$(a,5)=1$ is always going to be a necessary condition for any chosen $\alpha$; since we require $a$ to be relatively prime to $4$ consecutive numbers and $2\alpha-1$ in the statement.
An appendix contains supplementary information that is not an essential part of the text itself but which may be helpful in providing a more comprehensive understanding of the research problem or it is information that is too cumbersome to be included in the body of the paper.
\section{Quasi-Cyclic Form}\label{QCForm}
\begin{theorem}\label{th:QC}
Let $a\geq 5$ be an odd integer and $H$ be the parity-check matrix constructed in Equation \eqref{eq:H} based on the DM$(3;a)$ given by Equation \eqref{eqDM} with $\alpha=\frac{a-1}{2}$. Then there exists $H^*$ obtained from $H$ by row and column permutations that is the parity-check matrix of a $(a^2,4,a)$-regular binary QC-LDPC code.
\end{theorem}

We permute the rows and columns of the matrix $H$ constructed in Equation \eqref{eq:H}, based on the DM$(3;a)$ given by Equation \eqref{eqDM} to obtain $H^*$ which is a parity-check matrix of a QC-LDPC code. As we will use only row and column permutations the properties such as rate and the minimum distance of the code will not change.

First define the permutation $f$ on the columns $i$, $0\leq i\leq a^2$ of $H$ as $f(pa+q)=(q-p \mod\ a)a+p$ for all $0\leq p,q\leq a-1$. Then define the permutation $g$ on the rows $r_{i}$, $0\leq i\leq 4a-1$ of $H$ as \begin{eqnarray}\label{eq:DM}
g(r_i)=\left\{\hspace{-0.1cm}\begin{array}{ll}
r_i,&\mbox{if }r_i\leq 2a-1,\\
(r_i-2a)2^{-1}+2a,&\mbox{if }2a\leq r_i\leq 3a-1, \\
(r_i-3a)(\alpha+1)^{-1}+3a &\mbox{if }3a\leq r_i\leq 4a-1.\\
\end{array}
\right.
\end{eqnarray}
Then $g(f(H))=H^*$ is the parity-check matrix of a QC-LDPC code. To see this, let $C^*_{i}$ be the set of rows which have the entry $1$ in the column $i$ in the matrix $H^*$. Now as $f^{-1}(ap+q)= qa+ (p+q\mod a)$, it is not hard to compute that
\begin{eqnarray}
C^*_{ap+q}=\{q, (p+q \mod a)+a,\nonumber\\
(p.2^{-1}+q \mod a)+2a,\nonumber\\
(p(\alpha+1)^{-1}+q\mod a)+3a\}\nonumber
\end{eqnarray}
for $0\leq p,q\leq a-1$.

Hence after the permutations are applied the resulting matrix will have the following form:

\begin{eqnarray*}
H^*=\left[
\begin{array}{ccccccccccccc}
I&I&I&I&...&I\\
I&P^1&P^2&P^3&...&P^{a-1}\\
I&P^{2^{-1}}&P^{2*2^{-1}}&P^{3*2^{-1}}&...&P^{(a-1)*2^{-1}}\\
I&P^{(\alpha+1)^{-1}}&P^{2*(\alpha+1)^{-1}}&P^{3*(\alpha+1)^{-1}}&...&P^{(a-1)*(\alpha+1)^{-1}}\\
\end{array}\right].
\end{eqnarray*}

\begin{example}\label{ex:H25}
Let $a=5$. Then choose $\alpha=5-1/2=2$. In $\mathbb{Z}_5$ we have $2^{-1}=3$.
After the permutations are applied on $H$,  $H^*$ will have the form:
\begin{eqnarray*}
H^*=\left[
\begin{array}{ccccccccccccc}
I&I&I&I&I\\
I&P^1&P^2&P^3&P^4\\
I&P^3&P^1&P^4&P^2\\
I&P^2&P^4&P^1&P^3\\
\end{array}\right],
\end{eqnarray*}
which is
\begin{eqnarray*}
H^*=\left[
\begin{array}{ccccccccccccc}
I&I&I&I&I\\
I&P^1&P^2&P^3&P^4\\
I&P^{2^{-1}}&P^{2*2^{-1}}&P^{3*2^{-1}}&P^{4*2^{-1}}\\
I&P^{3^{-1}}&P^{2*3^{-1}}&P^{3*3^{-1}}&P^{4*3^{-1}}\\
\end{array}\right].
\end{eqnarray*}
\end{example}
Note that $3^{-1}=(\alpha+1)^{-1}$. 


\section{Parity-Check Matrices}\label{PCM}

\begin{figure}[ht!]
	\centering
	\includegraphics[width=\linewidth]{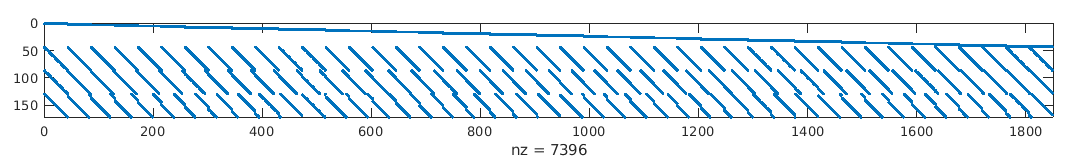}
	\caption{The parity-check matrix of the proposed code DM(3;43) $[1849, 0.91, 4, 43]$.}
	\label{fig:dm43}
\end{figure}
\begin{figure}[ht!]
	\centering
	\includegraphics[width=\linewidth]{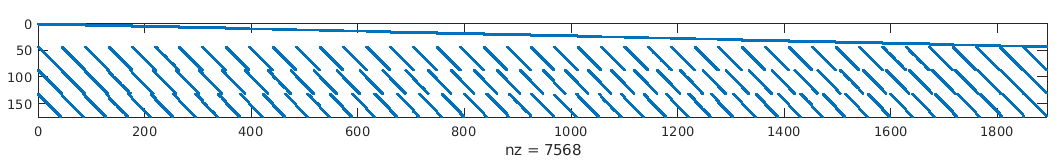}
	\caption{The parity-check matrix of the proposed code DCA(3;44) $[1892, 0.91, 4, 43]$.}
	\label{fig:dca44}
\end{figure}
\begin{figure}[ht!]
	\centering
	\includegraphics[width=\linewidth]{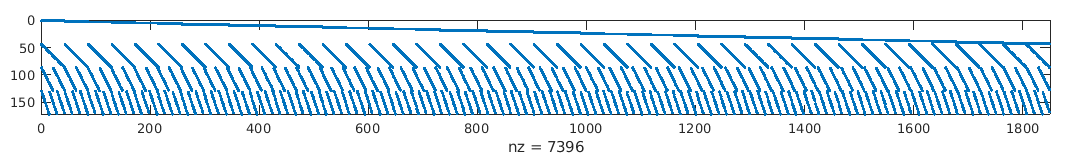}
	\caption{The parity-check matrix of TD-LDPC $[1849, 0.91, 4, 43]$. }
	\label{fig:td-ldpc}
\end{figure}
\begin{figure}[ht!]
	\centering
	\includegraphics[width=\linewidth]{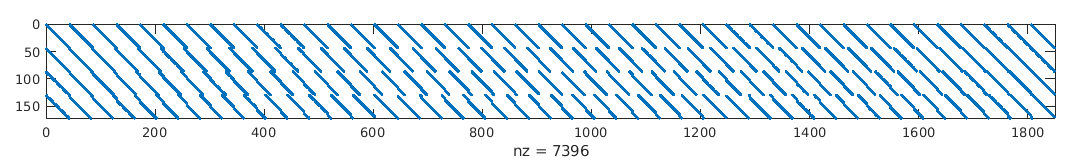}
	\caption{The parity-check matrix of Lattice $[1849, 0.91, 4, 43]$.}
	\label{fig:lattice}
\end{figure}
\begin{figure}[ht!]
	\centering
	\includegraphics[width=\linewidth]{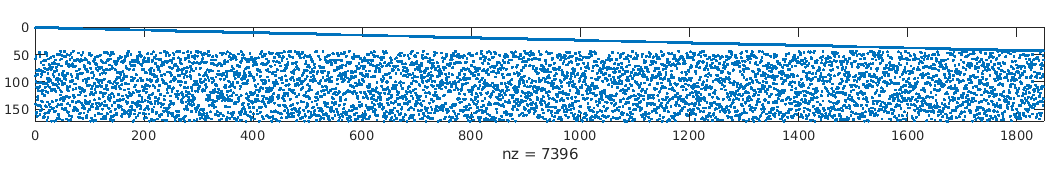}
	\caption{The parity-check matrix of Gallager $[1849, 0.91, 4, 43]$.}
	\label{fig:gallager}
\end{figure}
\begin{figure}[ht!]
	\centering
	\includegraphics[width=\linewidth]{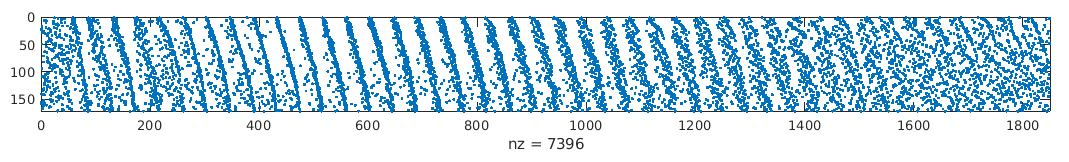}
	\caption{The parity-check matrix of PEG $[1849, 0.91, 4, \sim 43]$.}
	\label{fig:peg}
\end{figure}
\begin{figure}[ht!]
	\centering
	\includegraphics[width=\linewidth]{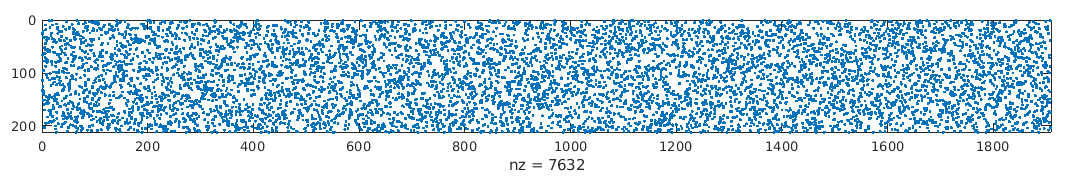}
	\caption{The parity-check matrix of Mackay\&Neal $[1908, 0.89, 4, 36]$.}
	\label{fig:mackay&neal}
\end{figure}
\end{appendices}


\bibliography{sn-bibliography}



\end{document}